\documentclass[reqno]{amsart}
\usepackage{fancyhdr}
\usepackage{booktabs}
\usepackage{amssymb, latexsym, amsthm, enumitem, color, amsmath}
\usepackage{array}
\usepackage{enumitem}
\usepackage{mathtools}
\usepackage{bbm}
\usepackage{nicefrac}
\usepackage{upquote}
\usepackage{amsmath}
\usepackage{tikz, tikz-cd, tikzsymbols, ifthen, circuitikz}
\usetikzlibrary{decorations.markings}
\usepackage[inkscapelatex=false]{svg}
\usepackage{graphicx}
\usepackage{subcaption}
\usepackage{svg}
\usetikzlibrary{arrows.meta, calc, matrix, graphs, graphs.standard}
\tikzset{
    marks/.style={decoration={markings,
    mark=at position 0.15 with {\fill[red] circle [radius=1.5pt];},
    mark=at position 0.3 with {\fill[red] circle [radius=1.5pt];},
    mark=at position 0.7 with {\fill[blue] circle [radius=1.5pt];},
    mark=at position 0.85 with {\fill[blue] circle [radius=1.5pt];}}, postaction={decorate}},
    nobreak/.style={decoration={markings, mark=at position 0.5 with {\draw[white,line width=20pt] (0,-5pt) -- (0,5pt);}}, postaction={decorate}},
    middots/.style={decoration={markings, mark=at position 0.5 with {\node[draw=none] {$\cdots$};}}, postaction={decorate}},
}
\usepackage{pgfplots}
\usepackage[unicode=true,pdfusetitle,bookmarks=true,bookmarksnumbered=false,
bookmarksopen=false,breaklinks=true,pdfborder={0 0 0},
pdfborderstyle={},backref=false,colorlinks=true]{hyperref}
\definecolor{myblue}{rgb}{0.09,0.32,0.44} 
\hypersetup{pdfborder={0 0 0},pdfborderstyle={},colorlinks=true,linkcolor=myblue,citecolor=myblue,urlcolor=blue}

\newtheorem{thm}{Theorem}[section] 
\newtheorem*{thm*}{Theorem}

\newtheorem{claim}[thm]{Claim}

\newtheorem{cor}[thm]{Corollary}
\newtheorem{defn}[thm]{Definition}

\newtheorem{fact}[thm]{Fact}
\newtheorem{lem}[thm]{Lemma}
\newtheorem{prop}[thm]{Proposition}

\newtheorem{question}[thm]{Question}
\newtheorem*{que*}{Question}
\newtheorem*{fact*}{Fact}

\newtheorem{rem}[thm]{Remark}

\newtheorem*{rem*}{Remark}
\newtheorem*{rems*}{Remarks}

\newcommand\Cref[1]{{Corollary~\ref{#1}}}

\newcommand{\N}{\mathbb{N}}
\newcommand{\Z}{\mathbb{Z}}
\newcommand{\R}{\mathbb{R}}

\newcommand{\E}{\mathbb{E}}

\newcommand{\ceil}[1]{\left\lceil #1 \right\rceil}

\newcommand{\argmin}{\mathrm{argmin}}
\newcommand{\osc}{\mathfrak{osc}}

\newcommand{\sph}{infinity superharmonic }
\newcommand{\Diam}{\mathrm{Diam}}
\newcommand{\en}{\mathcal{E}}

\newcommand{\hld}{H\"older }
\newcommand{\sign}[1]{\mathrm{sign}(#1)}

\renewcommand{\Pr}{\mathbb{P}}
\renewcommand{\sign}{\mathrm{sgn}}

\newcommand{\e}[1]{\exp(#1)}
\newcommand{\gs}{g^\sharp}
\newcommand{\lex}{\Psi}

\makeatletter
\def\moverlay{\mathpalette\mov@rlay}
\def\mov@rlay#1#2{\leavevmode\vtop{%
   \baselineskip\z@skip \lineskiplimit-\maxdimen
   \ialign{\hfil$\m@th#1##$\hfil\cr#2\crcr}}}
\newcommand{\charfusion}[3][\mathord]{
    #1{\ifx#1\mathop\vphantom{#2}\fi
        \mathpalette\mov@rlay{#2\cr#3}
      }
    \ifx#1\mathop\expandafter\displaylimits\fi}
\makeatother

\newlength{\tempindent} 
\newcommand{\lazyenum}{
\setlength{\tempindent}{\parindent} 
\begin{enumerate}[leftmargin=0cm,itemindent=0.7cm,labelwidth=\itemindent,labelsep=0cm,align=left,label=(\arabic*)]
\setlength{\parskip}{\smallskipamount}
\setlength{\parindent}{\tempindent}
}
\newif
\iffurther 
\furtherfalse

\makeatletter
\def\@settitle{\begin{center}%
  \baselineskip14\p@\relax
    \normalfont\LARGE
  \@title
  \end{center}%
}
\makeatother

 \title{Convergence rate of  $\ell^p$-energy minimization   on graphs: \\ sharp polynomial bounds and a phase transition at $p=3$}
\author{\Large Gideon Amir}

\address{G. Amir, Bar-Ilan University, Ramat Gan 52900, Israel}
\email{Gideon.amir@biu.ac.il}
\author{Fedor Nazarov}
\address{F. Nazarov, Department of Mathematical Sciences, Kent State University, Kent, OH, USA}
\email{nazarov@math.kent.edu
}
\author{Yuval Peres}

\address{Y. Peres, Beijing Institute of Mathematical Sciences and Applications (BIMSA), Huairou district, Beijing, China}
\email{yperes@gmail.com}

\makeatletter
\@namedef{subjclassname@2020}{2020 Mathematics Subject Classification}
\makeatother
\subjclass[2020]{}
 \pgfplotsset{compat=1.18}
\begin{document}

\begin{abstract}
 We consider the following dynamics on a connected graph $(V,E)$ with $n$ vertices. Given $p>1$ and an initial opinion profile $f_0:V \to [0,1]$, at each integer step $t \ge 1$ a uniformly random vertex $v=v_t$ is selected, and the opinion there is updated to the value $f_{t}(v)$ that minimizes the sum
 $\sum_{w \sim v} |f_t(v)-f_{t-1}(w)|^p$ over neighbours $w$ of $v$. 
 The case $p=2$ yields  linear averaging dynamics, but for all $p \ne 2$ the dynamics are nonlinear. In the limiting case $p=\infty$ (known as {\em Lipschitz learning}), $f_t(v)$ is the average of the largest and smallest values of $f_{t-1}(w)$ among the neighbours $w$ of $v$.
  We show that the number of steps needed to reduce the oscillation of $f_t$ below $\epsilon$ is at most $n^{\beta_p}$ (up to logarithmic factors in $n$ and   $\epsilon$), where  $\beta_p:=\max(\frac{2p}{p-1},3)$; we prove that the exponent $\beta_p$ is optimal.
 The phase transition at $p=3$ is a new phenomenon. We also derive matching upper and lower bounds for convergence time as a function of $n$ and the average degree; these are the most challenging to prove. 
\end{abstract}
\maketitle
\section{Introduction}
Let   $G=(V,E)$ be a finite connected graph\footnote{ All the graphs considered in this paper are
undirected and  simple  (no self-loops or multiple edges).}, where each vertex $v$ is assigned an initial opinion $f_0(v)\in[0,1]$.  
Given $1<p<\infty$, the asynchronous $\ell^p$\textit{-energy minimization} dynamics on $G$ are defined by choosing uniformly a vertex $v_t\in V$ at each step $t \ge 1$, and updating the value at $v_t$ to   minimize  the $\ell^p$-energy of $f_t$, leaving the values at other vertices unchanged:
\begin{equation}\label{eq:pdef}
    f_t(v) := \begin{cases}
        \argmin_{y} \sum_{w\sim v}|f_{t-1}(w)-y|^p \, \quad &v=v_t\, , \\
        f_{t-1}(v)  \quad &v\neq v_t \, .
    \end{cases}
\end{equation} 
The minimizing $y$ is unique, since  the function $x \mapsto |x|^p$ is strictly convex on $\mathbb R$.

The case $p=2$ yields  a linear averaging dynamics and is well understood; see Section \ref{s:background}.
For other $p$, the dynamics are nonlinear.

The dynamics can also be defined for $p=\infty$ by using the following update rule instead of \eqref{eq:pdef}:
\begin{equation}\label{eq:liplearning}
    f_t(v):= \begin{cases}
        \frac12\Bigl(\max_{u\sim v}f_{t-1}(u) + \min_{w\sim v}f_{t-1}(w)\Bigr) \, \quad &v=v_t \, ,\\
        f_{t-1}(v) \ \ \quad   &v\neq v_t\, .
    \end{cases}
\end{equation}
That is, we replace the value at $v_t$ by the average of the maximal and minimal values at its neighbours. This update rule is known as \textit{Lipschitz learning} (see, e.g., \cite{kyng2015algorithms}), since each update minimizes the local Lipschitz
constant at $v_t$.

The following is well known and easy to prove, see
Section \ref{sub:basic}.
\begin{fact}\label{fa:conv to cons}
    For every $1<p\leq \infty$, any finite connected graph $G=(V,E)$, and every initial profile $f_0$,
    almost surely the following limits exist and are equal:
    \begin{equation*}
        \forall u,v\in V \ \ \lim_{t\rightarrow \infty}f_t(v)\stackrel{a.s.}{=}\lim_{t\rightarrow \infty}f_t(u) \, .
    \end{equation*}
\end{fact}
\begin{que*}
    How fast do opinions converge to a consensus?
\end{que*}

We measure the distance to consensus using oscillations:  
\begin{equation}
    \osc(f):= \max_{v}f(v) - \min_w f(w) \, .
\end{equation}
The $\epsilon$-consensus time is defined by:
$$ \tau_p(\epsilon) := \min\{t\geq 0 :\, \osc(f_t)\leq \epsilon\}\, .$$

We will also use the $\ell^p$ energy: 

\begin{equation}\label{eq:energy def}
    \en_p(f):= \sum_{\{u,v\} \in E}  |f(u)-f(v)|^p \, .
\end{equation}

\subsection{A bird's eye view}


We begin with the case $1<p<\infty$. We give sharp bounds that exhibit a surprising phase transition at $p=3$.
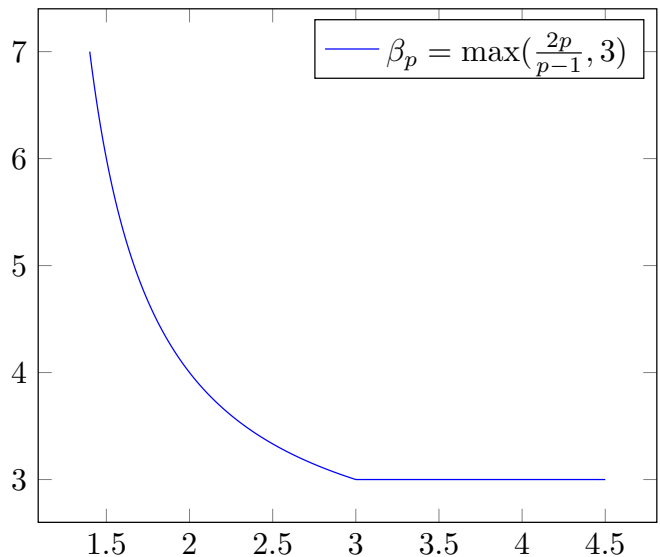
\begin{figure}[h]
    \centering
\begin{tikzpicture}[thick, scale = 1.2]
    \begin{axis}[domain=1.4:4.5,samples=400]
    
        \addplot+[mark=none] {max(2*x/(x-1), 3)};
        \addlegendentry{$\beta_p=\max(\frac{2p}{p-1},3)$}
    \end{axis}
\end{tikzpicture}
    \caption{The exponent $\beta_p$ that governs convergence to consensus in  Theorem \ref{th:lp no boundary simplified}. Note that $\frac{d}{dp}\beta_p\to -1/2$ as $p \uparrow 3.$}
    \label{fig:beta}
\end{figure} 

\begin{thm}\label{th:lp no boundary simplified}
    Fix $1<p<\infty$ and   define
    $\beta_p:=\max\Bigl\{\frac{2p}{p-1}, 3\Bigr\}$.
    There exist constants   $C_p,c_p>0$ 
    such that for any $n \ge 2$, any connected graph $(V,E)$ with $|V|=n$, and any   initial profile $f_0:V \to [0,1]$, the dynamics \eqref{eq:pdef} satisfy
\smallskip    
    \begin{itemize}
        \item[\bf{(a)}]
 $\E[\en_p(f_t)]\leq \e{-c_p  n^{-\beta_p}t}\en_p(f_0) \;$ for all $t>0$ {\rm ;}
  \smallskip    
    \item[\bf{(b)}] $\E[\tau_p(\epsilon)]\leq C_p n^{\beta_p}  \log \frac{n}{\epsilon} 
    \;$ for all $\epsilon \in (0,1/2]$.
    \end{itemize}
          Conversely, for some   $\tilde{c}_p>0$ and all  large $N$, there exist a connected graph $G=(V,E)$ with $|V|\leq N$ and an initial profile $f_0:V\rightarrow [0,1]$, such that 
\begin{equation}\label{eq:converse main} \tau_p(1/2)\geq  
    \tilde{c}_p N^{\beta_p}  \,,
\end{equation}
for any sequence of update vertices.
\end{thm}


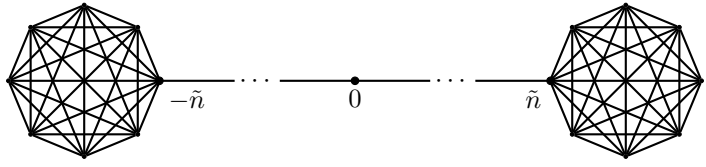
\begin{figure}\label{f:barbell}
    \centering
\begin{tikzpicture}
    \filldraw[black] (0,0) circle (0.05cm) node[anchor=north] {$0$};
    \foreach \s in {-1, 1} {
        \draw[thick] (0,0) -- (\s,0);
        \draw (\s*1.3, 0) node {$\cdots$};
        \draw[thick] (\s*1.6,0) -- (\s*2.6,0);
        \ifthenelse{\s=1}{\filldraw[black] (\s*2.6,0) circle (0.05cm) node[anchor=north east] {$\tilde{n}$}}{\filldraw[black] (\s*2.6,0) circle (0.05cm) node[anchor=north west] {$-\tilde{n}$}};
    }

    \def\n{8}
    \foreach \cent in {3.6, -3.6} {
    \node[circle,minimum size=2 cm] at (\cent,0) (b) {};
    \foreach\x in{1,...,\n}{
      \node[inner sep=0cm, minimum size=0.05cm, draw,circle, fill=black] (n-\x) at (b.{360/\n*\x}) {} ;
    }
    \foreach\x in{1,...,\n}{
      \foreach\y in{\x,...,\n}{
        \ifnum\x=\y\relax\else
          \draw[thick] (n-\x) edge[-] (n-\y);
        \fi
      }
      }
    }
\end{tikzpicture}
    \caption{The barbell graph with $n=4\tilde{n}  -1$ vertices: two cliques of size $\tilde{n}$ connected by a line segment of length $2\tilde{n}$. It provides the lower bound in Theorem \ref{th:lp no boundary simplified}, when the initial profile $f_0$ equals $0$ to the left of $0$, equals $1$ to the right of $0$, and satisfies $f(0)=1/2$.}
    \label{fig:barbell}
\end{figure}

 Theorem \ref{th:lp no boundary simplified} will follow from Theorem \ref{th:main lpbound}, which gives an upper bound on the energy in terms of the number of vertices $|V|$ and their average degree.  
 Graphs that are used to prove the lower bounds in \eqref{eq:converse main} are the  barbell   (see Figure \ref{fig:barbell}) for $1<p \le 3$, and the cycle for $p \ge 3$. More precise lower bounds are stated in the next section and proved in Section \ref{s:lowerbounds}.




Next we consider the case $p=\infty$. 

\begin{thm}\label{th:linfty no boundary simplified}
  {\bf (a)} Let $G=(V,E)$ be a connected graph with $|V|=n$. Run the Lipschitz learning dynamics \eqref{eq:liplearning} on $G$ starting with initial profile $f_0:V\to [0,1]$. Then for every $0<\epsilon<1$, the $\epsilon$-consensus time satisfies \begin{equation}\label{eq:lipbound}
      \E[\tau_{\infty}(\epsilon)] \leq n(\log n+1) (\Diam(G)+1)^2 \log \frac{2}{\epsilon}\, .
  \end{equation}
  {\bf (b)} Conversely, for every $N\geq L\geq 2$, there are a connected graph $G=(V,E)$ with $|V|\leq N$ and $\Diam(G)=L$, and an initial profile $f_0:V\to [0,1]$, such that 
  \begin{equation}\label{eq:liplower}
      \tau_{\infty}(1/2) \geq cN\cdot \Diam^2(G) \, 
  \end{equation}
   for any sequence of update vertices, where $c>0$ is an absolute constant.
\end{thm}
\begin{prop}\label{prop:syst}
    If the update vertices $\{v_t\}$ are chosen in  round robin fashion, that is, we update opinions at vertices selected according to a fixed cyclic
permutation of $V$, then the upper bound in \eqref{eq:lipbound} can be improved to
\begin{equation}\label{eq:lipsequpper}
         \tau_{\infty}(\epsilon) \leq n(\Diam(G)+1)^2 \log \frac{2}{\epsilon}\, . 
\end{equation}
\end{prop}
\begin{figure}[h]
    \centering
\begin{tikzpicture}[scale = 1.2]
    \node[circle, fill, inner sep=1pt, label={left:$0$}] (A) at (-3,0) {};
    \node[circle, fill, inner sep=1pt, label={right:$1$}] (B) at (3,0) {};
    \node[draw=none, fill=none] (C) at (0,0.1) {$\vdots$};

    \foreach \i in {-3,-2,-1,1,2,3} {
        \draw[marks,nobreak,middots] (A) to[out=18*\i, in=180-18*\i] (B);
    }
\end{tikzpicture}
\caption{A graph consisting of $k$ parallel paths of length $L$ between two nodes. This graph, with initial profile 0 on vertices closer to the left node and 1 on vertices closer to the right node, yields the lower bound in Theorem \ref{th:linfty no boundary simplified}.}
    \label{fig:lower necklace}
\end{figure}
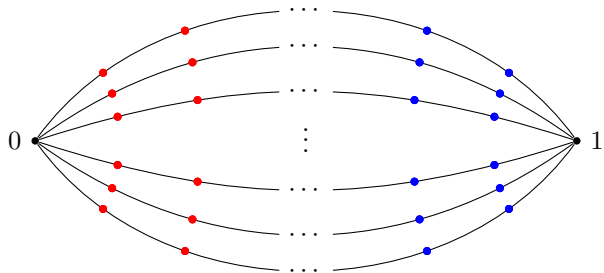

The upper bounds \eqref{eq:lipbound} and \eqref{eq:lipsequpper}
are proved in Section \ref{s:linfty no boundary}. The lower bound \eqref{eq:liplower} is shown in Section \ref{s:lowerbounds}; the relevant graph is in Figure \ref{fig:lower necklace}.


\subsection{Lipschitz learning with prescribed boundary values.}\label{sub:boundary}
Consider a finite connected graph $G=(V\cup B,E)$ where $B\neq \varnothing$ denotes the \textit{boundary} vertices and $V\neq \varnothing$ the \textit{interior} vertices. We assume that there are no edges between   vertices in $B$.  In this variation, given an initial profile $f_0:V\cup B\to [0,1]$, the sequence $\{f_t\}$ evolves according to \eqref{eq:liplearning}, where at each step $t$ the vertex $v_t$ is chosen uniformly from $V$. Thus, opinions at boundary vertices are never updated.

\begin{defn}
    We say that $h:V \cup B \to \R$ is {\bf infinity harmonic} on $V$ if \begin{equation}\label{eq:infinity harm}
        h(v)= \frac{\max_{u\sim v}h(u) + \min_{w \sim v}h(w)}{2}  \quad \, \forall v\in V \, . \end{equation}
\end{defn}
The existence of infinity harmonic extensions was first proved by Lazarus, Loeb, Propp, Stromquist and Ullman \cite{lazarus1999combinatorial} (See Proposition \ref{pr:infinity harmonic2}):
\begin{fact}\label{fa:infinity harmonic}
    Given a finite connected  graph $G=(V\cup B,E)$ and initial values $f_0$,  
 there is a unique extension $h:V\cup B \rightarrow \R$ of $f_0|_B$ that is infinity harmonic on $V$.
\end{fact}
 Efficiently finding this extension is the object of much study, see Section \ref{s:background}. It is well known and easy to see that under the Lipschitz learning dynamics, 
\begin{equation}\label{eq:infinity harmonic limit}
    \lim_{t\to \infty} f_t(v) = h(v) \quad \, \forall v\in V \, ,
\end{equation} where $h$ is the extension above. We will recall the proof of \eqref{eq:infinity harmonic limit} in Section \ref{sub:basic}. 

The $\epsilon$-approximation time is defined by:
\begin{equation}
    \tau^*(\epsilon):= \min\{t\geq 0 :\, \|f_t -h\|_\infty \leq \epsilon\}\, .
\end{equation}




The next theorem gives a tight polynomial bound on the mean of $\tau^*(\epsilon)$.
\begin{thm}\label{th:l infty with boundary simplified}
Let $G=(V\cup B,E)$ be a connected graph with boundary vertices $B$ and $|V|=n$.  Given an initial profile $f_0:V\cup B \rightarrow [0,1]$, let $h$ denote the infinity harmonic extension of $f_0$ from $B$ to $V$, and run the Lipschitz learning dynamics \eqref{eq:liplearning} with boundary values $f_0|_B$, as described in the beginning of this subsection. Then, 
\begin{itemize}
     \item[\bf{(a)}] the $\ell^1$ norm of the difference  $f_t-h$ satisfies 
\begin{equation}\label{eq:l1decrease} 
     \E[\|f_t- h\|_1]\leq n \,    e^{-t/(2n^3)} \,;
\end{equation}
\item[\bf{(b)}] the $\epsilon$-approximation time satisfies
\begin{equation}
    \E[\tau^*(\epsilon)]\leq 1+2n^3\log \frac{ne}{\epsilon} \,.
\end{equation}
\end{itemize}
Conversely, there is an absolute constant $c>0$ such that for every $n\geq 2$, on the segment of length $n$ with boundary values 1 at the endpoints and initial profile $0$ in the interior, we have
\begin{equation}
    \tau^*(1/2) \geq c n^3 \,,
\end{equation}
for any sequence of update vertices in $\{1,\ldots,n-1\}$.
\end{thm}

\begin{prop}\label{prop:round robin lip boundary}
    In the setting of Theorem \ref{th:l infty with boundary simplified}, if the update vertices $\{v_t\}$ are chosen in round robin fashion, then \eqref{eq:l1decrease} can be replaced by 
    \begin{equation*}
     \|f_{t}- h\|_1\leq n \exp\Bigl(\frac{-\lfloor t/n\rfloor}{2n^2}\Bigr) \,.
\end{equation*} 
Thus in this case,
\begin{equation*}
    \tau^*(\epsilon) \leq n+2n^3\log \frac{n}{\epsilon} \,.
\end{equation*}
\end{prop}

Theorem \ref{th:l infty with boundary simplified} and its extensions are proved in Section \ref{s:linftywithboundary}, while the relevant lower bounds are shown in Section \ref{s:lowerbounds}.
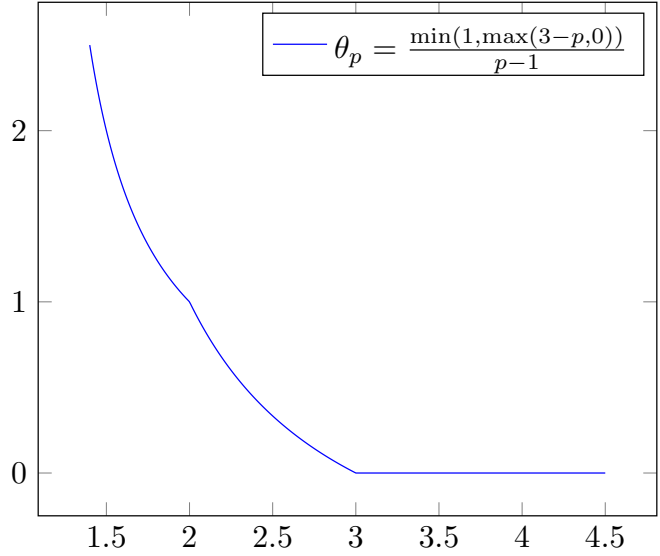
\begin{figure}[h]
    \centering
\begin{tikzpicture}[thick, scale = 1.2]
    \begin{axis}[domain=1.4:4.5,samples=400]
     legend style={cells={align=left}}
        \addplot+[mark=none] {min(1,max(0,3-x))/(x-1)};
         \addlegendentry{$\theta_p=\frac{{\rm min}(1,{\rm max}(3-p,0))}{p-1}$}
    \end{axis}
\end{tikzpicture}
    \caption{The exponent $\theta_p$ that governs the speedup of the convergence rate to consensus in  Theorem \ref{th:main lpbound}. Note that $\frac{d}{dp}\theta_p$ is discontinuous at $p=2$ and $p=3$, and that  $\frac{d}{dp}\theta_p\to -1/2$ as $p \uparrow 3.$}
    \label{fig:eta}
\end{figure} 

\subsection{Sharp bounds using the edge density}\label{s:detailed results}
Let $G=(V,E)$ be a connected graph with $|V|=n$ and 
 {\em average degree} $D_G:=\frac{1}{n}\sum_{v\in V}{\rm deg}(v) \, $.   Next we will state a more precise version of Theorem \ref{th:lp no boundary simplified} that uses the {\em edge density} $D_G/n=2|E|/n^2$. Recall that $\beta_p=\max\Bigl\{\frac{2p}{p-1},3\Bigr\}$
  and define
    \begin{equation}\label{eq:thetadef}
        \theta_p:=\frac{1}{p-1} \quad \text{\rm if} \quad  1<p \le 2 \quad \text{\rm and} \quad \theta_p:=\max\Bigl\{\frac{3-p}{p-1},0 \Bigr\} \quad \text{\rm if} \quad  p \ge 2 \,.
    \end{equation}  
    Consider the  function  
\begin{equation} \label{defF}
F(n,p,D):=n^{-\beta_p}(D/n)^{-\theta_p} =
 \begin{cases}
n^\frac{1-2p}{p-1}D^\frac{-1}{p-1} \, \ \, &1<p\leq 2 \, ,\\
n^{-3}D^\frac{p-3}{p-1} \, \ \, &2< p <3 \, ,\\
n^{-3} \, \ \, &p\geq 3 \, .    \end{cases}
\end{equation}

\begin{thm}\label{th:main lpbound}
    Fix $1<p<\infty$. There are constants   $c_p,C_p>0$ such that for all $n \ge 2$, for  every connected graph $G=(V,E)$ with $|V|=n$, and
    for  every initial profile $f_0:V \to [0,1]$, the dynamics \eqref{eq:pdef} satisfy
    \smallskip    
    \begin{itemize}
        \item[\bf{(a)}]
    $\E[\en_p(f_t)]\leq \en_p(f_0)\exp\bigl(-c_p F(n,p,D_G) t\bigr) \, $ for all $t>0$;
\medskip    
    \item[\bf{(b)}] 
    $ \displaystyle \E[\tau_p(\epsilon)]\leq \frac{C_p \log(n/ \epsilon)}{F(n,p,D_G)} 
    \;$ for all $\epsilon \in (0,1/2]$.
    \end{itemize}
    \medskip
Conversely, there exists  $\tilde{c}_p>0$ such that for   every  large $N$ and every $D\geq 2$, there are a connected graph $G=(V,E)$ with $|V|\leq N$ and $D_G \le D$, and an initial profile $f_0:V\rightarrow [0,1]$, such that 
\begin{equation}\label{eq:p D lower}
\tau_p(1/2) \ge 
\frac{\tilde{c}_p }{F(N,p,D)} \,.\textit{}
\end{equation}

\end{thm}
Theorem \ref{th:lp no boundary simplified} follows from the above theorem since $D_G\leq |V|$. 
Parts (a),(b) of Theorem \ref{th:main lpbound} are proved in Section \ref{s:lpbounds}. 
The convergence exponent $\theta_p$ exhibits phase transitions  at $p=2$ and $p=3$. The proofs of the upper bounds (a),(b), as well as  the extreme graphs we use to establish the lower bounds in Section \ref{s:lowerbounds}, change at these points.  

\begin{figure}[h]
    \centering

\begin{tikzpicture}[scale=1]
    \filldraw[black] (0,0) circle (0.05cm) node[anchor=north] {$0$};
    \foreach \s in {-1, 1} {
        \draw[thick] (0,0) -- (\s,0);
        \draw (\s*1.3, 0) node {$\cdots$};
        \draw[thick] (\s*1.6,0) -- (\s*2.6,0);
        \ifthenelse{\s=1}{\filldraw[black] (\s*2.6,0) circle (0.05cm) node[anchor=north east] {$n$}}{\filldraw[black] (\s*2.6,0) circle (0.05cm) node[anchor=north west] {$-n$}};
    }

    \def\n{6}
    \def\xx{4.5}
    \foreach \sc in {-1, 1} {
        \foreach \high in {1, -1, 2, -2} {
            \def\cent{\sc*\xx};
            \node[circle,minimum size=0.75 cm] at (\cent,\high) (b) {};
            \foreach\x in{1,...,\n}
            {
              \node[inner sep=0cm, minimum size=0.05cm, draw,circle, fill=black] (n-\x) at (b.{360/\n*\x}) {} ;
            }
            \foreach\x in{1,...,\n}{
                \foreach\y in{\x,...,\n}{
                    \ifnum\x=\y\relax\else
                    \draw[thick] (n-\x) edge[-] (n-\y);
                    \fi
                }
            \draw[red] (\sc*2.6, 0) -- (n-\x);
            }
        }
        \draw (\sc*\xx,0.1) node {$\vdots$};
    }
\end{tikzpicture}
    \caption{The graph $H_{6,n}$, a representative of the graphs $H_{d,n}$ used to prove the converse statement in Theorem \ref{th:main lpbound} for $1<p\leq 2$. There are $m=\frac{n}{6}$ cliques of size $6$ on each side. }    \label{fig:lp between 1 and 2}

\end{figure}
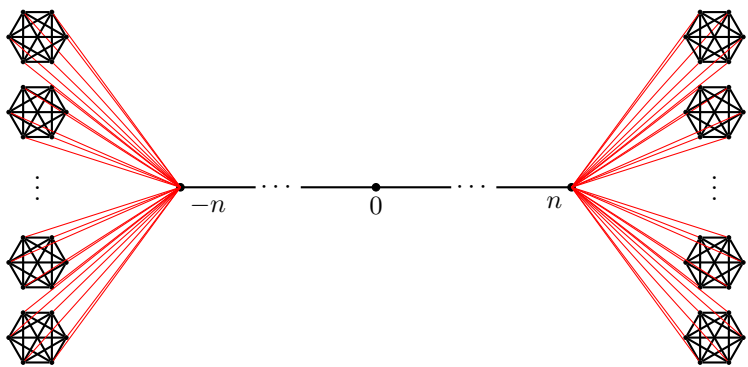

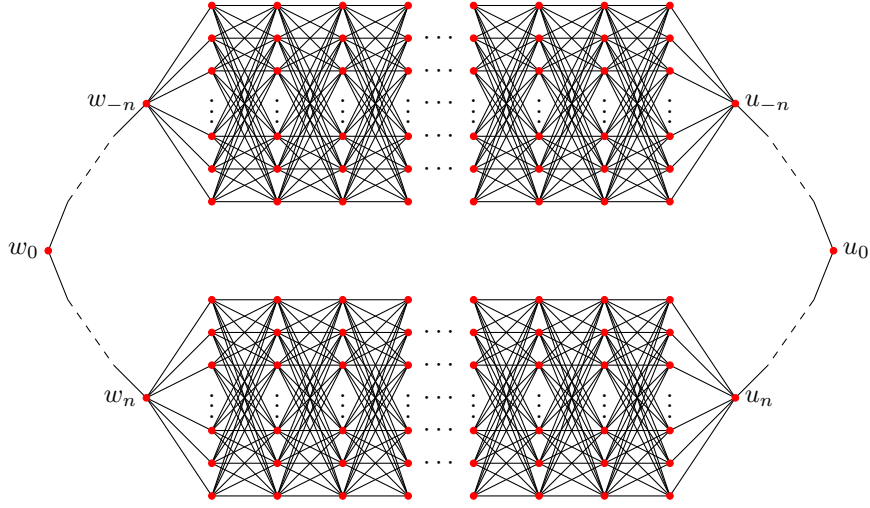
\begin{figure}[h!]
    \centering
\begin{tikzpicture}[scale=0.87]
  \def\goody{2.25};
  \foreach \sign in {1,-1} {
  \def\basey{ \sign*\goody };
  \foreach \x in {-3, -2, -1, 1, 2, 3} {
    \foreach \y in {-1.5, -1, -0.5, 0.5, 1, 1.5} {
        \foreach \yy in {-1.5, -1, -0.5, 0.5, 1, 1.5} {
          \draw (\x,\basey+\y) -- (\x+1,\basey+\yy);
        }
    }
  }

  \foreach \y in {-1.5, -1, -0.5, 0.5, 1, 1.5} {
    \draw (-4,\basey) -- (-3,\y+\basey);
  }
  \foreach \y in {-1.5, -1, -0.5, 0.5, 1, 1.5} {
    \draw (5,\basey) -- (4,\y+\basey);
  }
  \draw (-4.5,\basey-\sign*0.5) -- (-4,\basey);
  \draw (5.5,\basey-\sign*0.5) -- (5,\basey);
  \draw (-5.5,0) -- (-5.2,\sign*0.75);
  \draw (6.5,0) -- (6.2,\sign*0.75);
  \draw[dashed] (6.2,\sign*0.75) -- (5.5,\basey-\sign*0.5);
  \draw[dashed] (-5.2,\sign*0.75) -- (-4.5,\basey-\sign*0.5);

  \foreach \x in {-3, ..., 4} {
  \filldraw[black] (\x,\basey) circle (0cm) node {$\vdots$} ;
    \foreach \y in {-1.5, -1, -0.5, 0.5, 1, 1.5} {
        \filldraw[red] (\x,\basey+\y) circle (0.05cm) node[anchor=north] {} ;
    }
  }

  \foreach \y in {-1,-0.5,0,0.5,1}
    \filldraw[black] (0.5,\basey+\y) circle (0cm) node {$\cdots$} ;
  }

  \filldraw[red] (-4,\goody) circle (0.05cm) node[anchor=east, color=black] {$w_{-n}$} ;
  \filldraw[red] (5,\goody) circle (0.05cm) node[anchor=west, color=black] {$u_{-n}$} ;
  \filldraw[red] (-4,-\goody) circle (0.05cm) node[anchor=east, color=black] {$w_n$} ;
  \filldraw[red] (5,-\goody) circle (0.05cm) node[anchor=west, color=black] {$u_n$} ;

  \filldraw[red] (-5.5,0) circle (0.05cm) node[anchor=east, color=black] {$w_0$} ;
  \filldraw[red] (6.5,0) circle (0.05cm) node[anchor=west, color=black] {$u_0$} ;
\end{tikzpicture}
    \caption{The accordion graph, which is used to prove the lower bound in Theorem \ref{th:main lpbound} for $2\leq p \le 3$. Each of the top and bottom parts of the graph consists of $1+2n/d$ anti-cliques of size $d$ linked in a chain (where $d=\lfloor D/2 \rfloor$), with two paths of length $2n$ connecting them via the anchor nodes $w_{\pm n}$ and $u_{\pm n}$.}  
    \label{fig:accordion}
\end{figure}

\subsection{Background and history}\label{s:background}

The  $\ell^p$-energy minimization dynamics have been studied most intensively for $p=2$. In this case,  \eqref{eq:pdef}   updates the value at the selected vertex $v_t$ to the average of the values at its neighbours; in particular, the dynamics are linear. The  case $p=2$ is an asynchronous version of the    dynamics  introduced by deGroot~\cite{degroot1974reaching} as a model for non-Bayesian social learning (see also  the survey by Golub and Jackson \cite{golub2010naive}). In deGroot's original paper, the dynamics are synchronous, i.e., all vertices update their opinions simultaneously at each step, based on the current opinions.   DeGroot \cite{degroot1974reaching} and   Demarzo, Vayanos and Zwiebel \cite{demarzo2003persuasion} proved that if $G$ is not bipartite, then in the long run, opinions converge to consensus. Tight bounds on convergence rates for the asynchronous version of the deGroot dynamics were recently proved by Elboim, Peres and Peretz in \cite{elboim2022asynchronous} and are related to the spectral properties of the graph. In particular, Theorem $2.1$ (b) there implies that \[ \E[\tau_2(\epsilon)]\leq n^2 D_G \cdot  \Diam(G) \lceil \log_2(1/\epsilon)\rceil \leq n^3D_G \lceil \log_2(1/\epsilon)\rceil \, ,\]
and the RHS agrees with Theorem \ref{th:main lpbound} up to a $\log n$ factor.


Like the deGroot dynamics, one could study the synchronous $\ell^p$ energy minimization dynamics for any $1<p \leq \infty $, where at each time $t$ the opinions at \textit{all} vertices are updated simultaneously using the opinions at their neighbours. 
In the present work we restrict our attention to asynchronous updating.

The dynamics can also be considered  in continuous time, as was done for $p=2$  in \cite{elboim2022asynchronous}, by putting i.i.d.\ Poisson clocks on the vertices, and updating the value at a vertex when its clock rings.   All results that are stated in the present paper for the discrete time models can be easily translated to the continuous time models, by making a time-change. In particular, if $\tau^{\rm Cont}_p(\epsilon)$ denotes the $\epsilon$-consensus time of the continuous-time dynamics, then $\E\bigl[\tau^{\rm Cont}_p(\epsilon)\bigr]=\E\bigl[\tau_p(\epsilon)\bigr]/n$.

 
For the Lipschitz learning dynamics \eqref{eq:liplearning},  the update step is simple to calculate, but the absence of strict convexity changes the nature of arguments required to analyze the dynamics. 
Infinity harmonic functions arise as value functions for tug-of-war games analyzed by Peres, Schramm, Sheffield and Wilson \cite{peres2009tug}; they used this connection to generalize  Fact \ref{fa:infinity harmonic} (due to \cite{lazarus1999combinatorial}) on  existence of infinity harmonic extensions to infinite graphs and length spaces.   The game theoretic interpretation of the infinity harmonic extension immediately implies  that this extension operator is monotone: if the boundary values are increased, then the extension cannot decrease. This monotonicity can also be deduced from the convergence of the Lipschitz learning dynamics.  Lazarus et al. \cite{lazarus1999combinatorial} also presented a  polynomial time algorithm  for finding infinity harmonic extensions on finite graphs. Their algorithm was generalized to 
weighted graphs by Kyng, Rao, Sachdeva and Spielman \cite{kyng2015algorithms}. Oberman \cite{oberman2005convergent} suggested computing the infinity harmonic extensions in Euclidean space by first discretizing the problem and then solving the corresponding graph problem by an iterative method. 
To the best of our knowledge, no polynomial bounds on the convergence time of the dynamics \eqref{eq:pdef} and \eqref{eq:liplearning} have been previously obtained for any $p\neq 2$. 

We remark that the related problem of finding $p$-harmonic extensions on graphs was studied in the literature in the contexts of semi-supervised learning (see, e.g., Flores, Calder and Lerman \cite{flores2022analysis} and  Elmoataz, Desquesnes  and Toutain \cite{elmoataz2017game}) and of $\ell^p$-regression (e.g., by Adil, Kyng, Peng and Sachdeva in \cite{MR3909555}).

For $p=1$, the $\ell^1$ norm is no longer strictly convex and the minimizer in \eqref{eq:pdef} is no longer necessarily unique. Even when it is unique, the behaviour of the model is quite different than for $p>1$. For example, Figure \ref{fig:ell1 fixed point} gives a non-constant fixed point of the $\ell^1$ dynamics. 
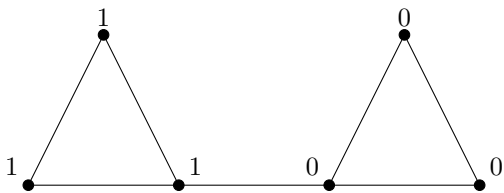
\begin{figure} 
    \centering
    \begin{tikzpicture}[>=stealth]

\coordinate (A) at (0,0);
\coordinate (B) at (-1,-2);
\coordinate (C) at (1,-2);

\draw (A) -- (B) -- (C) -- cycle;

\coordinate (D) at (4,0);
\coordinate (E) at (5,-2);
\coordinate (F) at (3,-2);

\draw (D) -- (E) -- (F) -- cycle;

\draw[-] (C) -- (F);

\foreach \vertex/\label in {A/1, B/1, C/1, D/0, E/0, F/0}
{
    \filldraw (\vertex) circle (2pt) node[above] {};
    
}

\node at (A) [above] {1};

\node at (D) [above] {0};
\node at (B) [above left=0.1mm] {1};

\node at (C) [above right=0.1mm] {1};

\node at (E) [above right=0.1mm] {0};

\node at (F) [above left=0.1mm] {0};

\end{tikzpicture}
    \caption{A fixed point of the $\ell^1$ energy minimization dynamics.}
    \label{fig:ell1 fixed point}
\end{figure}
\subsection{$p$-superharmonic functions and monotonicity of the dynamics.}\label{s:monotone}
Let $f:V\to \R$. For $1<p<\infty$, the dynamics \eqref{eq:pdef}
at time $t=1$, with initial opinion profile $f_0=f$ and update vertex $v=v_1$, yield  a new value $f_1(v)$ that minimizes
$\Psi_f(y)=\sum_{w \sim v} |y-f(w)|^p$. Thus
$$0=\Psi_f'(f_1(v))=p\sum_{w \sim v}|f_1(v)-f(w)|^{p-1}\sign(f_1(v)-f(w))\,.$$
We say that $f$ is $p$-{\bf superharmonic} at $v$ if $f(v) \ge f_1(v)$ when $f_1$ arises from an update at $v$. Since $\Psi_f$ is strictly convex, this is equivalent to
\begin{equation} \label{eq:psuper}
0 \le \Psi_f'(f(v))=p\sum_{w \sim v}|f(v)-f(w)|^{p-1}\sign(f(v)-f(w)) \,.
\end{equation}

This observation implies the known fact that the dynamics \eqref{eq:pdef}   are 
 monotone: 
\begin{claim} \label{claim-mono}
If the opinion profiles $f,g:V \to \R$  satisfy $f \le g$,  and   both profiles are updated at the same vertex $v$,
yielding $f_1$ and $g_1$ respectively, then $f_1(v) \le g_1(v)$. 
\end{claim}
\begin{proof}
For every $y \in \R$, we have
$$\Psi_f'(y)=p\!\!\sum_{w \sim v \atop  f(w)\leq y}\!\!\!(y-f(w))^{p-1}-p\!\!\sum_{w \sim v \atop f(w)>y}\!\!\!(f(w)-y)^{p-1} \,.$$
For each vertex $w \sim v$, when we replace $f(w)$ by $g(w)>f(w)$, the corresponding summand decreases (or becomes negative) if it is positive or zero, and increases in absolute value if it is negative. Thus 
$ \Psi_f'(y)\ge \Psi_g'(y)$ for all real $y$, whence $\Psi_f'(g_1(v)) \ge 0$, i.e., the function obtained from $f$ by replacing the value at $v$ by $g_1(v)$ is $p$-superharmonic at $v$. 
\end{proof}

\subsection{Sketch of convergence}\label{sub:basic}
Let us sketch why the convergence to consensus holds (Fact \ref{fa:conv to cons}). For this, the randomness assumption on the sequence of update vertices  $\{v_t\}_{t \ge 1}$ can be relaxed; it is only required that this  sequence  visits each vertex in $V$ infinitely often. 

For $1<p<\infty$, it is clear that the energy $\en_p(f_t)$ is non-increasing in $t$, and therefore must converge. We deduce that every subsequential limit of $\{f_t\}$ is a fixed point of the dynamics \eqref{eq:pdef}. Such functions are called $p$-{\bf harmonic} on $V$, and the usual proof of the maximum principle  shows that for $p>1$, any function that is $p$-harmonic at all vertices of a finite, connected graph is constant. Since $\min_{v \in V} f_t(v)$ is non-decreasing in $t$, we deduce that the opinions $\{f_t\}$ converge to a consensus. 
Our proof of Theorem  \ref{th:main lpbound}(a) is based on quantifying the expected  decrease   of the energy in each step.

The argument for $p=\infty$ is similar, once we find a potential function to replace the energy $\en_p(f)$. One such function is naturally suggested by the description in~\cite{kyng2015algorithms} of infinity harmonic extensions as lexicographic gradient minimizers (after  non-increasing rearrangement, as explained below). Given a function $f: V \to \mathbb R$ and and vertices $x,y$ such that  edge $e=\{x,y\} \in E$, define the {\bf gradient} $\nabla f(x,y):=f(y)-f(x)$ and the 
{\bf absolute gradient} 
$|\nabla f(e)|:=|f(y)-f(x)|$.
Enumerate the edges of $G$ as $\{e_i\}_{i=1}^m$, so that  the absolute gradients $\{|\nabla f(e_i)|\}_{i=1}^m$ are non-increasing, and define $\lex(f):=\sum_{i=1}^m |\nabla f(e_i)|3^{-i}$. If $f(v)$ differs from the average $\frac{f(v^+)+f(v^-)}{2}$ of the maximal and minimal neighbouring values, then moving $f(v)$ continuously towards this average decreases the largest absolute gradient at $v$, and could increase (at the same rate) only strictly smaller absolute gradients. It follows that $\lex(f_{t+1}) <\lex(f_t)$ for every $t$ such that $f_{t+1} \ne f_t$.  Thus the sequence $\{\lex(f_t)\}_{t \ge 0}$ must converge. Since the dynamics \eqref{eq:liplearning} and the mapping $f \mapsto \lex(f)$ are continuous, every subsequential limit of $\{f_t\}$ must be a fixed point of these dynamics, 
i.e, it is infinity harmonic on $V$. When there is no boundary, each subsequential limit  must be constant; the existence of   $\lim_t f_t$ follows from the monotonicity of $t \mapsto \min_v f_t(v)$ as before. 

Our proofs of Theorems \ref{th:linfty no boundary simplified}(a) and \ref{th:l infty with boundary simplified}(a) use better potential functions, for which we can obtain good estimates for the expected improvement over time. 


\section{  Convergence rates for Lipschitz learning}\label{s:linfty no boundary}

We first state a generalization of  Theorem \ref{th:linfty no boundary simplified} to arbitrary update sequences.

\begin{thm}\label{thm:liplearn}
     Let $G=(V,E)$ be a finite connected graph, and let $\{v_t\}_{t\geq 1}$ be an arbitrary sequence of vertices in $V$. 
     Define inductively a sequence of times $\{T_k\}$ by \begin{equation}  
T_0=0 \; \;  \text{ and} \quad   T_{k+1}=\inf\bigl\{t>T_k: \{v_j\}_{j=T_k+1}^t \; \text{covers} \; V\bigr\} 
\, ,
\end{equation}
    where the infimum of the empty set is $\infty$. Run the Lipschitz learning dynamics \eqref{eq:liplearning} on $G$ using the update sequence $\{v_t\}$, starting from initial profile $f_0:V\to [0,1]$. Then for every $k\geq 1$ and every $t\geq T_k$, we have  
     \[ \osc(f_t) \leq 2\exp\left( \frac{-k} {\Diam^2(G)+\Diam(G)}\right) \, .\]
     Therefore, \begin{equation}
         \tau_\infty(\epsilon)\leq T_{k(\epsilon)} \quad \text{where}  \quad k(\epsilon) := \ceil{\bigl(\Diam(G)^2+\Diam(G)\bigr) \log \frac{2}{\epsilon}} \, .
     \end{equation}
\end{thm}

Let   
$d_G: V \times V \to \N=\{0,1,2,\dots\}$ denote the graph distance in $G$. 
We say that a non-decreasing function $\omega: \N\rightarrow [0,\infty)$ that satisfies $\omega(0)=0$ is a \textbf{modulus of continuity} for  a function $f:V\rightarrow \R$,  if   $|f(u)-f(v)|\leq \omega(d_G(u,v))$ for all $u,v\in V$.
\begin{figure}
    \centering
        \includegraphics[width=1\columnwidth]{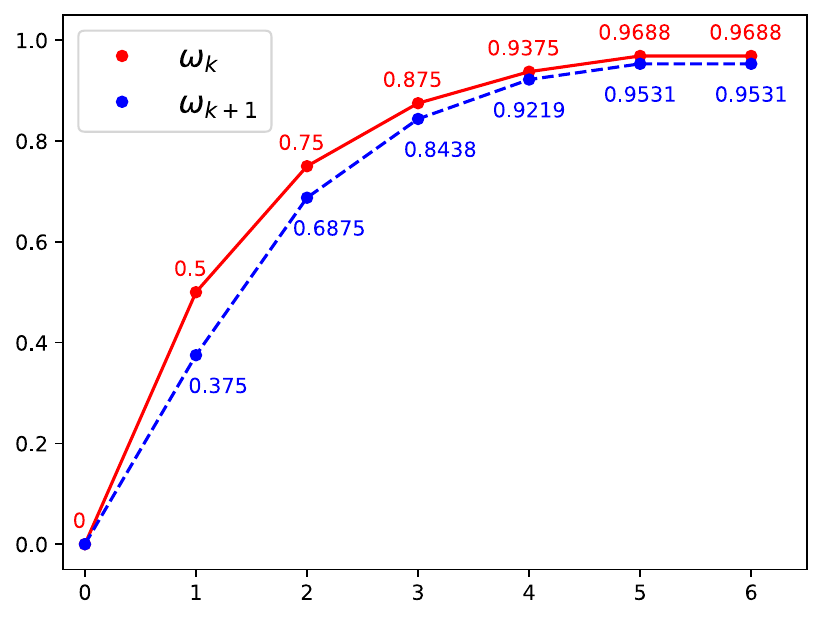}         \caption{An example of the transition from $\omega_k$ to
$\omega_{k+1}$ on a graph of diameter 5. }
    \label{fig:lip learn modulus example}
\end{figure}
The key to proving   Theorem \ref{thm:liplearn} is the following claim.
\begin{claim} \label{claim-conc0}
Let $L=\Diam(G)$ in the graph metric. Suppose that  $\omega$ is a   modulus of continuity for $f: V \to {\mathbb R}$ that satisfies $\omega(r)=\omega (L)$ for all $r\geq L$. If a Lipschitz learning update is performed at $v \in V$ leading to the profile $f^v$,  then  for every $u\in V\setminus \{v\}$, we have 
\[|f^v(u)-f^v(v)|\leq \frac{\omega(d_G(u,v)-1)+\omega(d_G(u,v)+1)}{2}\,.\] 
\end{claim}
\begin{proof}
Fix $u\in V\setminus \{v\}$ and let   $v^u$ be a neighbour of $v$ lying on a shortest path from $v$ to $u$. Denote by $v^+$  (respectively, $v^-$) a neighbour  of $v$ that maximizes (respectively, minimizes) $f$.
Then
\begin{equation*}
   f^v(v)= \frac{f (v^+)+f (v^-)}{2} \le  
     \frac{f (v^+)+f (v^u)}{2} \, .
\end{equation*}
Let $r=d_G(u,v)$, so $d_G(v^{\pm},u)\leq r+1$ and $d_G(v^u,u)=r-1$. Then 
\begin{equation}\label{eq:up1}
    f^v(v)-f^v(u)\le    \frac{f (v^+)-f (u)+f (v^u)-f (u)}{2} \le \frac{\omega(r+1)+\omega(r-1)}{2}\, .
\end{equation}
Similarly, since
\begin{equation*}
 f^v(v)\ge \frac{f(v^u)+f(v^-)}{2}   \, ,
\end{equation*}
we infer that
\begin{equation}\label{eq:down1}
    f^v(u)-f^v(v)\le     \frac{f(u)-f(v^u)+f(u)-f(v^-)}{2} \le \frac{\omega(r-1)+\omega(r+1)}{2}\, .
\end{equation}
The inequalities \eqref{eq:up1} and \eqref{eq:down1} complete  the proof of the claim.
\end{proof}

\begin{lem} \label{lemma-conc}
In the setting of Theorem \ref{thm:liplearn},  suppose that  $\omega_k$ is a concave modulus of continuity for $f_{T_k}$ that satisfies $\omega_{k}(r)=\omega_k(L)$ for $r\geq L$. Let $\omega_{k+1}(0)=0$ and define 
\begin{equation}\label{eq:def wk}
\omega_{k+1}(j)=\begin{cases}
    \frac{\omega_k(j-1)+\omega_k(j+1)}{2}  \ \ 1\leq j\leq L \, ,\\
    \omega_{k+1}(L) \ \  j>L \,.
\end{cases}   \end{equation} 
Then \begin{itemize}
\item[(a)] $\omega_k$ is a  modulus of continuity for $f_s$ for all $s \ge T_k$ . \\
\item[(b)] If the opinion at $v=v_t$ was updated at some time $t \in (T_k,s]$, then for every $u\in V\setminus \{v\}$, we have
\[|f_s(u)-f_s(v)|\leq 
 \omega_{k+1}(d_G(u,v)) \, .
 \] 
\item[(c)] $\omega_{k+1}$ is a concave modulus of continuity for $f_{T_{k+1}}$.
\end{itemize}
\end{lem}
\begin{proof} 
We prove parts (a) and (b) together by induction on $s \ge T_k$. The base case is clear, and the induction step from $s-1$ to $s$ follows from Claim \ref{claim-conc0} and the concavity of $\omega_k$.

The definition of $\omega_{k+1}$ implies that it inherits the monotonicity and concavity properties from $\omega_k$. 
Part (c) follows, since from time $T_k$ to  time $T_{k+1}$, each vertex was updated at least once.
 \end{proof}
\begin{figure}
    \centering
    \includegraphics[width=1\columnwidth]{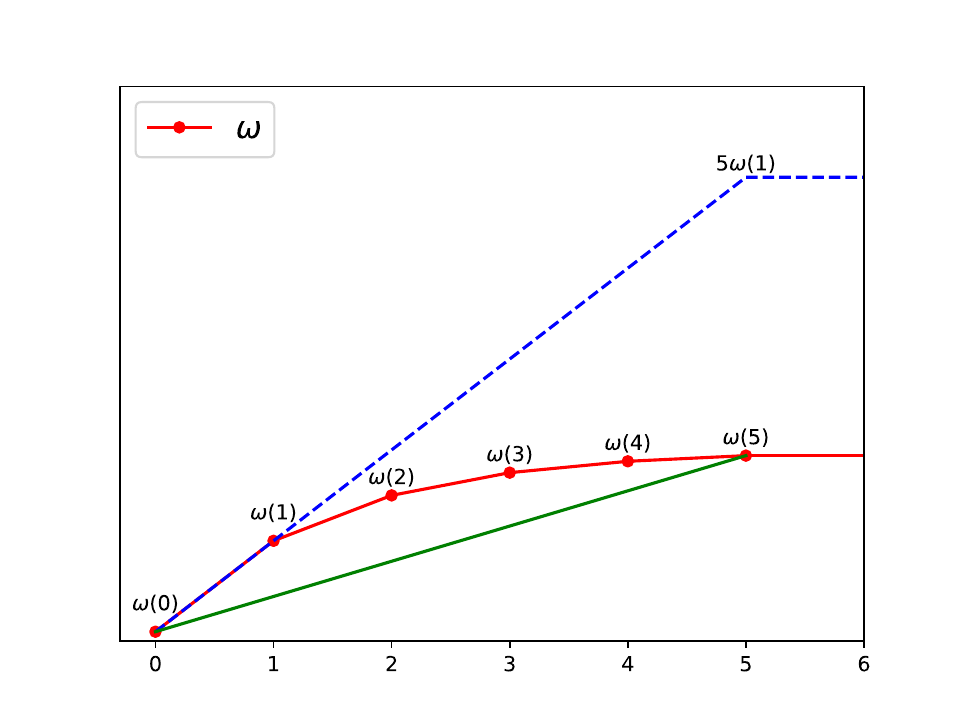}         \caption{The modulus of continuity $\omega$ is bounded on $[0,L]$ between the two straight lines of slopes       $\omega(L)/L$ and $\omega(1)$, respectively.}
    \label{fig:lip learn modulus steps}
\end{figure}

\begin{proof}[Proof of Theorem \ref{thm:liplearn}]
Given a function $\omega:\N \rightarrow [0,\infty)$ that satisfies $\omega(0)=0$, let $S(\omega):=\sum_{i=0}^{L} \omega(i)$. If $\omega$ is also concave and non-decreasing, then  
\[\frac{L \omega(L)}{2}\leq S(\omega)\leq \frac{L (L+1)}{2}\omega(1).\] 

The function $\omega_0(j)={\bf 1}_{\{j>0\}}$ is a concave modulus of continuity for the initial profile $f_0: V \to [0,1]$. For $k \ge 1$, define $\omega_k$ recursively via \eqref{eq:def wk} and 
let $S_k:=S(\omega_k)$. Then by  \eqref{eq:def wk} we have \begin{equation*}
    S_{k+1}\leq S_k - \frac12 \omega_k(1)\leq S_k \left(1-\frac{1}{L (L+1)}\right)\,,
\end{equation*}
and therefore 
\begin{equation*} 
  \omega_k(L) \leq \frac{2}{L} S_k\leq \frac{2}{L} \left(1-\frac{1}{L (L+1)}\right)^k S_0  = 2 \left(1-\frac{1}{L (L+1)}\right)^k \, .    
\end{equation*}
Since $\osc(f_t)$ is non-increasing in $t$ and  $1-x\leq e^{-x}$, we deduce that for $t \ge T_k$,
\[ \osc(f_t)\leq \osc(f_{T_k})\leq \omega_k(L) \leq 2\exp\Bigl(-k/\bigl(\Diam^2(G)+\Diam(G)\bigr)\Bigr)\, , \]
completing the proof of the theorem.

\end{proof}

\begin{proof}[\bf{Proof of Theorem \ref{th:linfty no boundary simplified} (a)}]
   First note that $T_k-T_{k-1}$ are simply coupon collector times and therefore 
    \[\E[T_k]=k\E[T_1]\leq kn(\log n+1).\]
    Set $k=\ceil{\bigl(\Diam(G)^2+\Diam(G)\bigr) \log \frac{2}{\epsilon}}$.
   By the previous theorem, for  every $t\geq T_k$ we have 
   $ \osc(f_t) \leq \epsilon$,  and therefore \begin{equation*}
       \E[\tau_{\infty}(\epsilon)]\leq \E[T_k] = k\E[T_1] \leq k n(\log n+1)\leq n(\log n+1) (\Diam(G)+1)^2 \log \frac{2
       }{\epsilon}\, .
   \end{equation*}
\end{proof}
\begin{proof}
    [\bf{Proof of Proposition \ref{prop:syst}}]
    When using round robin updates, $T_k=nk $. We conclude from Theorem \ref{thm:liplearn} that $$\tau_\infty(\epsilon)\leq n
    (\Diam(G)+1)^2\log \frac{2}{\epsilon}\, .$$  
\end{proof}


\section
{Lipschitz learning with prescribed boundary values}\label{s:linftywithboundary}

\subsection{The infinity harmonic extension.}
Fix $G=(V\cup B,E)$.
Given a path $\gamma=(\gamma_0\rightarrow \gamma_1\rightarrow \gamma_2\rightarrow \dots \rightarrow \gamma_\ell)$ in $G$
and a real-valued function $h$ defined at the endpoints of $\gamma$, the \textbf{slope} of $h$ on $\gamma$ is $\frac{h(\gamma_\ell)-h(\gamma_0)}{\ell}$. 
We say that the simple path $\gamma=(\gamma_0\rightarrow \gamma_1\rightarrow \gamma_2\rightarrow \dots \rightarrow \gamma_\ell)$ of length $\ell \ge 2$  is a \textbf{bridge} if  $\gamma_0,\gamma_\ell\in B$ and $\gamma_1,\gamma_2,\dots,\gamma_{\ell-1}\in V$.

We now state a more detailed version of Fact \ref{fa:infinity harmonic}:
\begin{prop}\label{pr:infinity harmonic2}\cite{lazarus1999combinatorial}
    Let $G=(V\cup B,E)$ be a finite connected graph, where $V, B \ne \varnothing$. 
    Given  boundary values $h_B:B\to \R$,  there exists a unique extension $h:V\cup B \rightarrow \R$ of $h_B$ that is infinity harmonic on $V$. The extension $h$ is linear along every bridge $\gamma=(\gamma_0\rightarrow \gamma_1\rightarrow \gamma_2\rightarrow \dots \rightarrow \gamma_\ell)$ of maximal slope for $h_B$, that is, $h(\gamma_{k+1})-h(\gamma_k)$ equals the slope for all $0\leq k < \ell$.
\end{prop}
The proof of the proposition in \cite{lazarus1999combinatorial} is constructive. For convenience of the reader we include  an exposition of this proof, which uses the following lemma.
\begin{lem}\label{le:bridge slope}
    Let $G=(V\cup B,E)$ be a finite connected graph, where $V, B \ne \varnothing$. Suppose that  $h:V \cup B \to \mathbb R$ is infinity harmonic on $V$, and $e=\{v,w\}$ is an edge of $G$ with at least one endpoint in $V$. If $h(w)>h(v)$, then there is a  bridge $\gamma$  in $G$  that includes $e$, such that all the gradients of $h$ on $\gamma$ satisfy $h(\gamma_i)-h(\gamma_{i-1}) \ge  h(w)-h(v) $. In particular, the slope of $h|_B$ on $\gamma$ is at least  $ h(w)-h(v) $.

\end{lem}
\begin{proof}
      If $w \in V$, denote by $w^{+ }$ a neighbour of $w$ in $V \cup B$ where  $h$ is maximized. 
    Since $h$ is infinity harmonic on $V$, 
    \begin{equation} \label{slopeq}
       h(w^+)-h(w ) = h(w)-h(w^-) \ge h(w )-h(v)>0\,.
    \end{equation}
    Continuing in this manner, we obtain a simple path $\gamma_+$ from $w$ to $B$ where all the gradients of $h$ are at least $h(w )-h(v)$.
    Similarly, for $v \in V$, let  $v^-$   be a neighbour of $v$ that minimizes $h$, and note that   $$h(v^-)-h(v)= h(v)-h(v^+)  \le h(v)-h(w)\,.$$  Continuing recursively, we obtain a simple path $\gamma_-$ from $v$ to $B$ where all the gradients are at  most $h(v)-h(w)<0$.
     Concatenating the reversal of $\gamma_-$, the oriented edge $v \to w$, and the path $\gamma_+$,
    we obtain a bridge   $\gamma$ where all the gradients of $h$ on $\gamma$  are at least $h(w)-h(v)$, as claimed. 
\end{proof}

\begin{proof}[Proof of Proposition \ref{pr:infinity harmonic2}]
We will use induction on $|V|$. Let $W$ be a connected component of $V$.
We separate two cases.

\noindent{\bf Case 1}. If $W$ is adjacent to a single node  $b_W  \in B$, then the usual proof of the maximum principle shows that the unique extension of $h_B$ to $W \cup B$ which is infinity harmonic on $W$ is obtained by defining $h (w)=h_B(b_W)$ for every $w \in W$.
If $W=V$, then we are done, otherwise we can replace $V$ and $B$ by $V \setminus W$ and $B \cup W$, respectively, and apply the induction hypothesis.  

\noindent{\bf Case 2}.
If $W$ is adjacent to at least two nodes  in $B$, then there is a bridge that intersects $W$. To verify the uniqueness of the extension,
 suppose that $h: V \cup B \to \mathbb R$ is an infinity harmonic extension of $h_B$, and $\gamma=(\gamma_0 \to \dots \to \gamma_\ell)$  is a bridge intersecting $W$ such that $h_B$ has maximal slope $s \ge 0$ among such bridges. Then Lemma \ref{le:bridge slope} implies that $h(\gamma_k)-h(\gamma_{k-1}) \le s$ for all $1 \le k \le \ell$; since the average of these $\ell$ gradients is $s$, we conclude that they are all equal to $s$. Adding the interior nodes of $\gamma$ to the boundary, and removing them from $V$, we infer the uniqueness of the infinity harmonic extension  by induction on $|V|$.
\medskip

 We can use the same argument to construct  the infinity harmonic extension of $h_B$. First define $h$ by linear interpolation on a bridge  $\gamma=(\gamma_0 \to \dots \to \gamma_\ell)$ intersecting $W$  where $h_B$ has maximal slope $s$. Then use the induction hypothesis  to extend $h$ to the rest of $V$, so it is infinity harmonic on $V \setminus\{\gamma_1,\ldots, \gamma_{\ell-1}\}$. It remains to check that $h$ is infinity harmonic on 
 $\{\gamma_1,\ldots ,\gamma_{\ell-1}\}$. If not, then there is some $k \in [1,\ell)$ such that either $h(\gamma_{k}^+) >h(\gamma_{k+1})$  or $h(\gamma_{k}^-) <h(\gamma_{k-1})$. If the first of these inequalities holds,  then    we use Lemma  \ref{le:bridge slope} to obtain a simple path $ {\widetilde{\gamma}}$ from $\gamma_k$ to some vertex $z \in B \cup \{\gamma_1,\ldots ,\gamma_{\ell-1}\}$ where $h$ has slope strictly greater than $s$. 
 Let ${\Gamma}$ be the concatenation of $\gamma_0 \to\dots \to \gamma_k$ and $ {\widetilde{\gamma}}$. If $z \in B$, then ${\Gamma}$ is a bridge intersecting $W$ in $G$, where $h_B$ has  slope strictly greater than $s$, a contradiction; otherwise, $z=\gamma_j$ for some $j \in \{k+2, \dots, \ell-1\}$, and  ${\Gamma}$ followed by the path $\gamma_j\to \dots \to \gamma_\ell$ is a bridge in $G$ which yields the same contradiction.  The  case where $h( \gamma_{k}^-) <h(\gamma_{k-1})$ is similar.
\end{proof}

\subsection{Convergence rates for Lipschitz learning with boundary.}
Denote by $\Delta_\infty$   the {\em infinity Laplacian} on $G=(V\cup B,E)$, which maps functions \mbox{$f:{V \cup B} \to \R$} to real valued functions on $V$ via
$$(\Delta_\infty f)(v)= f(v^+)
+f(v^-)-2f(v) \,,$$ 
where $v^+$ (respectively, $v^-$) denotes the neighbour of $v$ in $V \cup B$ at which $f$ attains its highest (respectively,  lowest) value.

A function $f: V\cup B \rightarrow \R$ is   {\em infinity superharmonic} at $v\in V$ if $(\Delta_\infty f)(v)\leq 0$;  it is infinity harmonic at $v$ iff \mbox{$(\Delta_\infty f)(v)=0$.}

Recall (Fact \ref{fa:infinity harmonic}), that given $h_B: B \to \R$, there exists a unique  extension \\ $h:V\cup B \rightarrow \R$ of $h_B$ which is infinity harmonic on $V$. This extension is invariant under the Lipschitz learning dynamics. Furthermore, by the monotonicity  of the dynamics,   
if $f_0(v)\geq h(v)$ for all $v\in V\cup B$, then $f_t(v)\geq h(v)$ for all $t$ and all $v\in V\cup B$.
Note also that if $f_0$ is \sph on $V$, then $f_t$ is \sph on $V$ for all $t\geq 0$. 

The following lemma lies at the heart of our analysis.  
\begin{lem} \label{lem:keylip}  
 Suppose that the functions $f,h:V \cup B \to \R$ satisfy $f|_B=h|_B$. If $h$ is infinity harmonic on $V$ and $f$ is infinity superharmonic on $V$, then 
\begin{equation} \label{eq:inftyto1}
        \|f-h\|_\infty \leq n\|\Delta_\infty f\|_1 \, ,
    \end{equation}
    where $n=|V|$ and $\|\Delta_\infty f\|_1=\sum_{v \in V} |(\Delta_\infty f)(v)|$. Therefore, 
    \begin{equation} \label{eq:1to1}
        \|f-h\|_1 \leq n^2\|\Delta_\infty f\|_1 \, .
    \end{equation} 
\end{lem}

\begin{proof} 
Define $\delta:=\|\Delta_\infty f\|_1$. We will prove \eqref{eq:inftyto1} by induction on the number of interior vertices $n=|V|$.

We say that a path $\gamma = (\gamma_0\rightarrow \gamma_1 \rightarrow  \ldots \rightarrow \gamma_k)$  is  {\bf greedy} (for $f$) if $\{\gamma_j\}_{j=1}^{k-1}$ are distinct vertices in $V$,  and for every $1\leq i<k$, we have  $\gamma_{i+1}=\gamma_i^+$, that is, $f(\gamma_{i+1})=\max_{u\sim \gamma_i}f(u)$.
For any three adjacent vertices
$\gamma_{i-1},\gamma_i, \gamma_{i+1}$
along a greedy path,  
$$\nabla f(\gamma_{i-1},\gamma_i) =f(\gamma_i)-f(\gamma_{i-1}) \le  f(\gamma_i)-f(\gamma_i^-) \,,$$
so
$$ \nabla f(\gamma_{i-1},\gamma_i)-\nabla f(\gamma_{i},\gamma_{i+1})  \le 2f(\gamma_i)-f(\gamma_{i+1})-f(\gamma_i^-)=|\Delta_\infty f(\gamma_i)| \, . $$
Therefore,  if $e_1,e_2$ are oriented edges along a greedy path where $e_1$ precedes $e_2$, then 
\begin{equation} \label{eq:greedgrad} \nabla f(e_1)-\nabla f(e_2) \leq \delta= \|\Delta_\infty f\|_1\,.
\end{equation}

Next, we separate two cases:

\noindent{\bf Case 1}. Suppose  that every connected component $W$ of $V$ is adjacent to a single vertex $b_W \in B$. 

Fix such a component $W$. 
Then $h(w)=h(b_W)$ for all $w \in W$. If $w_0 \in W \cup\{b_W\}$
and $w_1 \in W$ are neighbours with $f(w_1)>f(w_0)$, then we construct   a greedy path $w_0 \to w_1 \to \dots \to w_\tau$ for $f$, where $\tau$ is minimal such that $f(w_\tau) \le f(w_{\tau-1})$. We infer from  \eqref{eq:greedgrad} that
$$ \nabla f(w_0,w_1) \le  \nabla f(w_0,w_1)- \nabla f(w_{\tau-1},w_\tau) \le \delta \,.$$ 
For every vertex $w \in W$, summing the gradients of $f$ along a shortest path from $b_W$ to $w$ yields $f(w)-h(w)=f(w)-f(b_W) \le n\delta$. This verifies \eqref{eq:inftyto1}  in Case 1.

\medskip

\noindent{\bf Case 2}.
If some connected component  of $V$ is adjacent to two or more vertices in $B$, then there exist bridges in $G$.
Choose a bridge $\gamma^*=(\gamma_0\rightarrow \gamma_1\rightarrow \gamma_2\rightarrow \dots \rightarrow \gamma_\ell)$ where $h|_B$ has maximal slope (See Figure \ref{fig:maxslope}). Recall from Proposition \ref{pr:infinity harmonic2} that the gradient of $h$ along $\gamma^*$ is constant, that is,  
     \[ \forall k\in [1,\ell], \quad h(\gamma_k)-h(\gamma_{k-1})=\frac{h(\gamma_\ell)-h(\gamma_0)}{\ell} \,.\]    
Since $\gamma_0,\gamma_\ell\in B$, we have $f(\gamma_0)=h(\gamma_0)$ and $f(\gamma_\ell)=h(\gamma_\ell)$. 
    Let \begin{equation*}
        \beta:=\frac{h(\gamma_\ell)-h(\gamma_0)}{\ell} \, \ \text{and} \, \ \alpha:=\max_{k\in [1, \ell-1]} \{f(\gamma_k)-h(\gamma_k)\}\ \,. 
    \end{equation*}
    The main step of the proof will be showing that \begin{equation} \label{eq:alphadelta}
        \alpha\leq (\ell-1)\delta \,.
    \end{equation}
    
    By a simple pigeonhole argument \[\exists j\in [1,\ell-1]\  \text{s.t.} \  \left(f(\gamma_j)-h(\gamma_j)\right) - \left(f(\gamma_{j-1})-h(\gamma_{j-1})\right)\geq \frac{\alpha}{\ell-1}.\]  
    Since $h$ has constant gradient $\beta$   on the path $\gamma^*$, we have 
    \begin{equation}\label{eq:gradientlowerbound}
     f(\gamma_j)-f(\gamma_{j-1})\geq \frac{\alpha}{\ell-1}+ \beta \,.
    \end{equation}

    We now construct a greedy path for $f$ starting from the oriented edge $\gamma_{j-1}\rightarrow \gamma_j$: we set $\gamma^0_j:=\gamma_j$ and for $k\geq 0$, we let $\gamma^{k+1}_j := (\gamma_j^k)^+$. 
    The path continues for $\tau$ steps, where
    \begin{equation} \label{eq:deftau} 
    \tau:=\min\{k \ge 1: \, f(\gamma_j^k)\leq f(\gamma_j^{k-1})   \quad \text{or}\quad \gamma_j^k\in B \} \,.
    \end{equation}

        \begin{figure}[htbp]
    \centering
    \begin{subfigure}[t]{0.5\textwidth}
        \includegraphics[width=\textwidth]{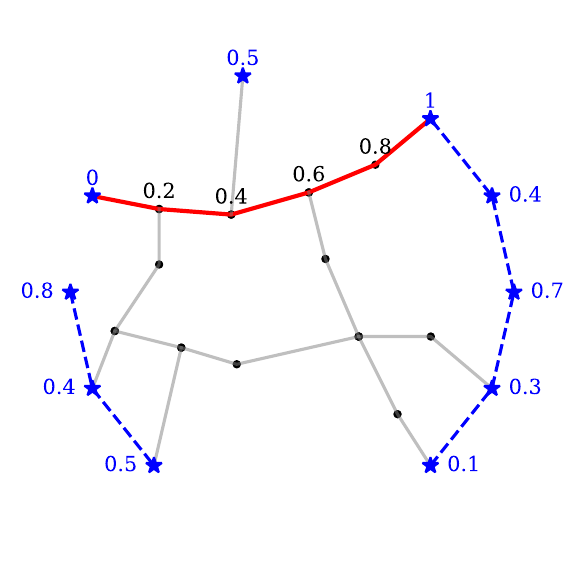}
        \vspace*{-17mm}
        \captionsetup{width=12cm}
        \caption{Boundary vertices are marked by stars. The bridge of maximal slope for $h$ is marked by the thick red line, and the values of $h$ on this bridge are shown.}
        \label{fig:maxslope}
    \end{subfigure}
    \begin{subfigure}[t]{0.5\textwidth}
        \includegraphics[width=\textwidth]{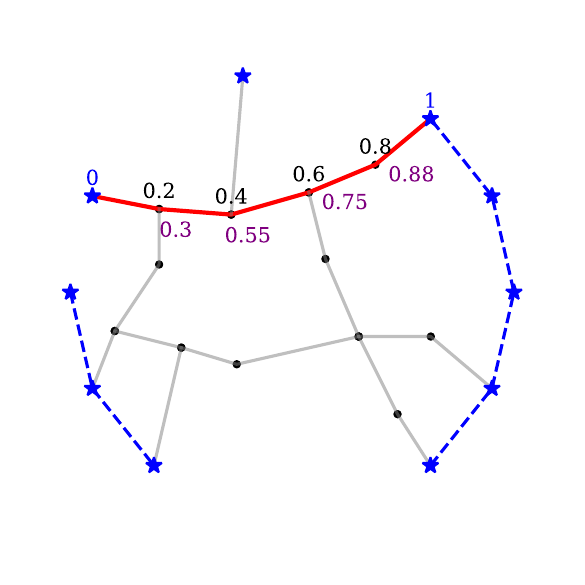}
        \vspace*{-17mm}
                \captionsetup{width=12cm}
        \caption{Values of $f$ on the bridge of maximal slope for $h$ are written below the vertices.}
        \vspace{-10mm}
        \label{fig:fvals}
    \end{subfigure}
        
    \begin{subfigure}[b]{0.5\textwidth}
        \centering
        \includegraphics[width=\textwidth]{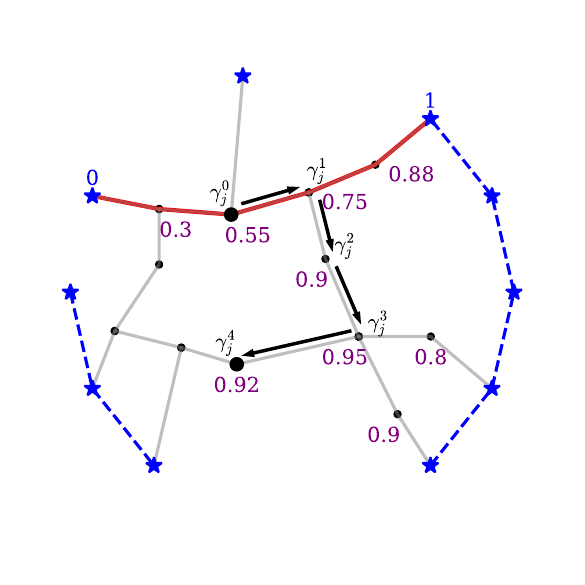} 
        \vspace*{-17mm}
                \captionsetup{width=12cm}
        \caption{The greedy path from $\gamma^0_j$, stopped at the first descent.}
        \label{fig:greedy1}
    \end{subfigure}
            \captionsetup{width=12cm}
    \caption{The construction used in the proof of Lemma \ref{lem:keylip}: we start with the bridge $\gamma^*$ of maximal slope for $h$ (\ref{fig:maxslope}), then   find an edge in $\gamma^*$  where $\nabla(f-h)$ is large (\ref{fig:fvals}). Finally, we build a greedy path for $f$ from that edge, stopping at the first descent or upon reaching the boundary (\ref{fig:greedy1}).}
    \label{fig:all_figures}
\end{figure}

The vertices 
$\{\gamma_j^k\}_{k=0}^{\tau-1}$ are all distinct, so $\tau \le |V|$.
      We will prove the inequality $\alpha \le (\ell-1)\delta$ separately in the two (overlapping) cases (see Figure \ref{fig:all_figures}):
    \begin{itemize}
    \item[(2a)]\label{cl:max} $f(\gamma_j^\tau)\leq f(\gamma_j^{\tau-1}) $, i.e.,  $f$ stopped increasing along the path;
    \item[(2b)] \label{cl:boundary} $\gamma_j^\tau\in B$.\\
\end{itemize}
 
    In case (2a),
    $ \nabla f(\gamma_j^{\tau-1},\gamma_j^\tau)\leq 0 \, $ and by \eqref{eq:gradientlowerbound}, \[\nabla f(\gamma_{j-1},\gamma_j)\geq \frac{\alpha}{\ell-1}.\] Since the path $\gamma_{j-1} \to \gamma_j^0 \to \dots \to \gamma_j^\tau$ is greedy, \eqref{eq:greedgrad} implies that $\frac{\alpha}{\ell-1}\leq \delta$, as required. This concludes case (2a).

    Next, we consider case (2b), where $\gamma_j^\tau \in B$. Since 
    $\gamma_{0} \to \dots \to \gamma_j  \to \gamma_j^1  \dots \to \gamma_j^\tau$ is a bridge and the slope of $h$ along any bridge is at most $\beta$,  we obtain that
    \begin{equation} \label{eq:slopeh} h(\gamma_j^\tau)-h(\gamma_0)\leq (\tau+j) \beta \, .
    \end{equation}
    Since $h$ has constant gradient $\beta$ on $\gamma^*$, the inequality \eqref{eq:slopeh} implies 
    that 
    \[h(\gamma_j^\tau)-h(\gamma_j)\leq \tau \beta \, .\]
    Using  the fact that $f(\gamma_j^\tau)=h(\gamma_j^\tau)$, as $\gamma_j^\tau \in B$, and the inequality $f\geq h$, we infer that \begin{equation}
        f(\gamma_j^\tau)-f(\gamma_j)\leq \tau \beta \, .
    \end{equation}
    Therefore, there is some $k\leq \tau$ for which $f(\gamma_j^k)-f(\gamma_j^{k-1}) \leq  \beta $. 
    In view of \eqref{eq:gradientlowerbound}, we conclude that \[\nabla f(\gamma_{j-1},\gamma_j) - \nabla f(\gamma_j^{k-1}, \gamma_j^k) \geq \frac{\alpha}{\ell-1}\, ,\] which     implies  that the inequality
$ {\alpha} \le (\ell-1)\delta$  
    holds in case (2b) as well.  
\medskip
    Set $\widetilde{B}=B\cup \{\gamma_1,...,\gamma_{\ell-1}\}$, and let $\widetilde{h}$ be the infinity harmonic extension of $f$ from $\widetilde{B}$ to $B \cup V$. Then $h\leq \widetilde{h}\leq h+\alpha$ and  $f\geq \widetilde{h}$. We can now conclude that \begin{equation*}
        \|f-h\|_\infty \leq \alpha + \|f-\widetilde{h}\|_\infty\leq (\ell-1)\delta +\bigl(n-(\ell-1)\bigr)\delta=n\delta \,,
    \end{equation*}
    with the last inequality coming from \eqref{eq:alphadelta} and  the induction hypothesis. This completes Case 2 and proves the theorem. 
  
\end{proof}

    

{\bf{Proof of Theorem \ref{th:l infty with boundary simplified} (a) and (b)}.}
\begin{proof}
By the monotonicity property, it is enough to prove the theorem for the two functions $f_0$ and $g_0$, defined on $V \cup B$ as follows: $f_0(v)=g_0(v)=h(v)$ for $v \in B$, while $f_0(v)=1$  and $g_0(v)=0$ for $v\in V$. We will focus on the upper envelope $f_0$, as  the statement for $g_0$ can be inferred by observing that $1-g_0$ is an upper envelope for $1-h$ that agrees with it on $B$, and $1-h$ is infinity harmonic on $V$.  Note that $f_0$ is infinity superharmonic, and therefore so are $f_t$ for all $t>0$.

Suppose that $f:V\cup B \to \mathbb R$ is infinity superharmonic and  agrees with $h$ on $B$. Let $f^v$ be obtained from $f$ by updating at the vertex $v$. Then 
\begin{equation} \label{eq:linfty update v} 
\|f-h\|_1 - \|f^v-h\|_1 = f(v)-f^v(v) =\frac{|\Delta_\infty f(v)|}{2} \, .\end{equation}

Averaging over the choice of the update vertex $v$, we get
\begin{equation}
   \|f-h\|_1- \frac1n\sum_{v \in V}\|f^v- h\|_1 \geq   \frac{1}{2n}\|\Delta_\infty f\|_1 \geq \frac{1}{2n^3}\|f-h\|_1 \,,
\end{equation}
where the rightmost inequality follows from \eqref{eq:1to1}.
We deduce that
 
Rearranging and applying this to $ f_s$ in place of $f$  yields
$$\E\bigl[\, \|f_{s+1}- h\|_1 \,  \mid \,f_s \bigr]\leq \Bigl(1-\frac{1}{2n^3}\Bigr) \|f_s-h\|_1  \,.
$$
Taking expectations and iterating this inequality for $s \in [0,t)$, we obtain
\begin{equation*}
     \E[\|f_t- h\|_1]\leq \Bigl(1-\frac{1}{2n^3}\Bigr)^t \|f_0-h\|_1 \leq n \Bigl(1-\frac{1}{2n^3}\Bigr)^t \,,
\end{equation*}
since $0 \le f_0, h \le 1$ on $V$ and $f_0=h$ on $B$.
Part (a) of the theorem now follows, as $(1-x)\leq e^{-x}$. 

To prove part (b), note that 
\begin{equation} \label{eq:tailb}
    \Pr \Bigl( \|f_t-h\|_\infty > \epsilon \Bigr) \leq \Pr \Bigl( \|f_t-h\|_1 > \epsilon \Bigr)    \leq \frac{n}{\epsilon} \exp\Bigl(-\frac{t}{2n^3}\Bigr) \,,
\end{equation}
where the second inequality follows from part (a) and Markov's inequality.  
The RHS of \eqref{eq:tailb} is less than 1 if 
$$t>t^*_\epsilon:=2n^3 \log \frac{n}{\epsilon} \,.$$
Therefore,
\begin{equation*}
\begin{split}
\E[\tau^*(\epsilon)] &=
\int_0^\infty \Pr \Bigl( \|f_{\lfloor t \rfloor} -h\|_\infty > \epsilon \Bigr)  \,dt \le 
t^*_\epsilon+1+\int_{t^*_\epsilon+1}^\infty \frac{n}{\epsilon} \exp\Bigl(-\frac{t-1}{2n^3}\Bigr)\, dt \\[1ex]
&=t^*_\epsilon+1+\int_{t^*_\epsilon+1}^\infty    \exp\Bigl(-\frac{t-1-t^*_\epsilon}{2n^3}\Bigr)\, dt=t^*_\epsilon+1+2n^3=1+2n^3 \log \frac{ne}{\epsilon} \, .
\end{split}   
\end{equation*}

\end{proof}

\begin{proof}[\bf{Proof of Proposition \ref{prop:round robin lip boundary}}]
As in the previous proof, it suffices to consider an infinity superharmonic $f_0$ that agrees with $h$ on $B$ and takes value 1 on $V$. Then  all the functions $f_s$ for $s \ge 0$ are infinity superharmonic, and satisfy $f_{s+1} \le f_s$. The key observation is that if $t \le s<t+n$ and $w \ne v_r$ for $r \in [t+1,s]$, then $|\Delta_\infty f_{s}(w)| \ge |\Delta_\infty f_{t}(w)|$. 
Therefore, we can infer from \eqref{eq:linfty update v} that
\begin{eqnarray*}
    \|f_{t}-h\|_1 - \|f_{t+n}-h\|_1 = \sum_{s=t}^{t+n-1} \frac{|\Delta_\infty f_{s}(v_{s+1})|}{2} \geq \sum_{s=t}^{t+n-1} \frac{|\Delta_\infty f_{t}(v_{s+1})|}{2} = \frac{\|\Delta_\infty f_t\|_1}{2} \, .
\end{eqnarray*}

In conjunction with \eqref{eq:1to1}, this yields  
$$ \|f_{t+n}-h\|_1  \le \Bigl(1-\frac1{2n^2}\Bigr) \|f_{t}-h\|_1 \,.$$
Thus,
\begin{equation*}
     \|f_{t}- h\|_1\leq \|f_0-h\|_1 \exp\Bigl(\frac{-\lfloor t/n\rfloor}{2n^2}\Bigr) \,.
\end{equation*}
Since   $\|f_0-h\|_1 \le n$, the first claim of the theorem is proved. The claimed upper bound on  $\tau^*(\epsilon)$ follows readily.
\end{proof}






\section{Convergence rates for $1<p<\infty$ with no boundary}\label{s:lpbounds}

Our main objective in this section is to prove parts (a) and (b) of Theorem \ref{th:main lpbound} on energy decay and time to $\epsilon$-consensus under the $\ell^p$-energy minimization dynamics. 

Recall the function    
\begin{equation}
F(n,p,D):=n^{-\beta_p}(D/n)^{-\theta_p} =
 \begin{cases}
n^\frac{1-2p}{p-1}D^\frac{-1}{p-1} \, \ \, &1<p\leq 2 \, ,\\
n^{-3}D^\frac{p-3}{p-1} \, \ \, &2< p <3 \, ,\\
n^{-3} \, \ \, &p\geq 3 \, ,    \end{cases}
\end{equation}
defined in \eqref{defF}. Define the constant $c_p$:
\begin{equation}
    c_p: = \begin{cases}
         p2^\frac{-2}{p-1} \, \ \; &1<p\leq 2 \, , \\
         \frac{p}{80(p-1)} \, \ \; &2<p \, .\\
    \end{cases}
\end{equation}
Note that $c_p\geq \frac{1}{80}$ for all $p> 2$, and that $c_p\rightarrow 0$ as $p\downarrow 1$.

Let $g:V \rightarrow \R$ be a  non-constant function, and suppose that $g^\sharp$ is  obtained from $g$ by updating the value at some vertex. 
As noted in the introduction,  $\en_p(\gs) \leq \en_p(g)$. The next lemma gives a lower bound on the relative energy decrease.
\begin{lem}\label{le:main p}
    Let $g:V \rightarrow \R$ be any non-constant function, and let $\gs$ be the function after updating the value at a uniformly chosen random vertex. 
    Then \begin{equation}\label{eq:ratio}
        \E\left[\frac{\en_p(g)-\en_p(\gs)}{\en_p(g)}\right] \geq c_p F(n,p,D) \, .
    \end{equation}
\end{lem}
For ease of notation, we identify the vertices of $G$ with the set $\{1,2,\cdots,n\}$ and write $g_i=g(i)$. We may assume, without loss of generality, that the vertices are ordered so that $g_1\geq g_2\geq \cdots\geq g_n$. 

Define the local $\ell^p$-energy of $g$ at a vertex $i$ by  
\begin{equation}
    \en_{i,p}(g) = \sum_{j: j\sim i} |g_i-g_j|^p \,.
\end{equation}
The (total) energy is
\begin{equation}
\en_p(g) = \frac12 \sum_i \en_{i,p}(g) \, .
\end{equation}

To show \eqref{eq:ratio},  we define
\begin{eqnarray} \label{eq:Rplus}
    R_i^+&:=&\sum_{j<i:j\sim i}(g_j-g_i)^{p-1} \,,\\ \label{eq:Rminus}
    R_i^-&:=&\sum_{j>i: j\sim i}(g_i-g_j)^{p-1}\, ,\\
    \rho_i&:=&|R_i^+ - R_i^-| \,, \\ R_i&:=&R_i^+ + R_i^- \quad  \text{and} \quad r_i:=R_i^\frac{1}{p-1} \,.
\end{eqnarray}

The derivative of $\en_p(g)$ with respect to $g_i$ is $p(R_i^- - R_i^+)$.
One can think of $pR_i^+$ as the ``upward pull" of the neighbours of $i$ on $g_i$, and of $pR_i^-$ as the ``downward pull".

Denote by $d_i$ the degree of the vertex $i$, write $\delta_i:={d_i^\frac{1}{p-1}}$, and define \begin{equation}\label{eq:ii}
    I_i= \begin{cases} 
     \rho_i^\frac{p}{p-1} \delta_i^{-1}  \; \quad &1<p\leq 2 \, ,\vspace{1ex}\\ 
     \rho_i^2\, r_i^{2-p} \, \delta_i^{-1}\; \quad  &p> 2\, .    
    \end{cases}
\end{equation}

Next, we give  a lower bound on the  energy decrease when vertex $i$ is updated.
\begin{claim}\label{cl:p improve}
If vertex $i$ is updated, then the energy decrease $\en_p(g)-\en_p(g^\#)$ is at least $\tilde{c}_p I_i$, where $\tilde{c}_p = c_p$ for $1<p\leq 2$ and $\tilde{c}_p=10c_p$ for $p> 2$. 
\end{claim}
\begin{proof}

We will show that there is some $y$ so that changing the value at $i$ from $g_i$ to $y$ results in the required decrease in the energy. The claim will follow since at an update the value chosen minimizes the energy.

Let $g_i(\lambda):=g_i \pm \lambda$, where we take $+\lambda$ if $R_i^+\geq R_i^-$, and $-\lambda$ otherwise. Let $\en_p(\lambda):=\en_p(g(\lambda))$, where in $g(\lambda)$ all coordinates except $g_i(\lambda)$ agree with $g$. Then the derivatives of $\en_p(\lambda)$ with respect to $\lambda$ satisfy
 
\begin{equation} \label{deriv0}
     \en_p'(0)  = -p \rho_i\,, 
\end{equation}
while 
\begin{eqnarray} \label{eq:second-der}
     \en_p''(\lambda)=p(p-1)\sum_{j:j\sim i}|g_i-g_j \pm \lambda|^{p-2} \,.
\end{eqnarray}
To lower bound   the  energy decrease when we replace $g_i$ by $g_i\pm \Lambda$,   we will derive an upper bound on $\en_p''(\lambda)$  for   $\lambda \in [0,\Lambda]$ and then use the identity 
 \begin{equation} \label{eq:newton}
      \en_p'(\Lambda)  = 
      \en_p'(0)+ \int_0^\Lambda  \en_p''(\lambda)d\lambda = -p\rho_i+\int_0^\Lambda  \en_p''(\lambda)d\lambda  \, .
 \end{equation}
 
    We divide the rest of the proof into two cases according to the value of $p$.\\
    
\noindent{\bf Case} $1<p\leq 2$:\\
    We use the easy inequality that for every $\Lambda>0$ and $a\in \R$, 
    \begin{equation}
            \int_0^\Lambda (p-1)|a \pm \lambda|^{p-2}d\lambda \leq 2 ( \Lambda/2)^{p-1} = 2^{2-p}\Lambda^{p-1} \,  
    \end{equation}
    (Since the maximum is achieved when $a=\mp \frac{\Lambda}{2}$).
    Summing the above inequalities over all $j\sim i$ with $a_j=g_i-g_j$, we obtain from \eqref{eq:second-der} that
\begin{equation}\label{eq:changeinderivative}
        \int_0^\Lambda  \en_p''(\lambda) \, d\lambda\leq p2^{2-p}\Lambda^{p-1}d_i \, .
    \end{equation}
     If we take 
    \begin{equation}
\Lambda^*:=\left(\frac{\rho_i}{2^{3-p}d_i}\right)^{\frac{1}{p-1}} \, ,
    \end{equation}
    then for all $\Lambda \in [0, \Lambda^*]$, the integral in \eqref{eq:changeinderivative} is at most $p\frac{\rho_i}{2}$, so \eqref{eq:newton}
    yields that $ \en_p'(\Lambda) \leq -p\frac{\rho_i}{2}$ for all $\Lambda \in [0,\Lambda^*]$, and therefore the energy decrease    when we replace $g_i$ by $g_i + \Lambda^*$ (if $R_i^+\geq R_i^-$) or by $g_i-\Lambda^*$ (if $R_i^+< R_i^-$), is at least \begin{equation}
       \Lambda^* \cdot p\rho_i/2 =\tilde{c}_pI_i \, ,
    \end{equation}
    where $\tilde{c}_p= p 2^{-\frac{2}{p-1}}$.
    This concludes the case $1 < p \leq 2$.
     \medskip
    
        \noindent{\bf Case $p > 2$}:\\
    Writing $R_i(\lambda):= \sum_{j:j\sim i}|g_i(\lambda)-g_j|^{p-1}$ and $r_i(\lambda)=R_i(\lambda)^{\frac{1}{p-1}}$, we have 
    \begin{equation}
         \en_p''(\lambda)=p(p-1)\sum_{j:j\sim i}|g_i(\lambda)-g_j|^{p-2}\leq p(p-1) r_i(\lambda)^{p-2} \delta_i  \, ,
    \end{equation}
    where the inequality follows from applying \hld with $\widetilde{p}=p-1$ and $\widetilde{q}=\frac{p-1}{p-2}$, and recalling that $\sum_{j:j\sim i} 1 =d_i$. Similarly,  
    \begin{equation}
         R_i'(\lambda)\leq (p-1)\sum_{j:j\sim i}|g_i(\lambda)-g_j|^{p-2}\leq (p-1)    r_i(\lambda)^{p-2} \delta_i \,.
    \end{equation}
Since $R_i'(\lambda)=(p-1)r_i(\lambda)^{p-2}r_i'(\lambda)$, we infer that
\begin{equation}
     r_i'(\lambda) \leq \delta_i \,,
\end{equation}
which implies that $r_i(\lambda)\leq r_i+\lambda \delta_i$. 
Thus, for  $\lambda \in [0,\Lambda]$, we have 
\begin{equation}\label{eq:second deriv p big}
     \en_p''(\lambda)\leq p(p-1)\delta_i  (r_i +\delta_i  \Lambda)^{p-2} \,.
\end{equation} 
    We now choose \begin{equation}
        \Lambda^*=\frac{\rho_i}{4(p-1)r_i^{p-2}\delta_i} \leq \frac{r_i}{4(p-1)\delta_i}  \, ,
    \end{equation} 
where  the inequality holds because $\rho_i\leq R_i=r_i^{p-1}$.
Then for all $\Lambda \in [0, \Lambda^*]$,
\begin{equation}\label{eq:Rit bound}
    ( r_i + \delta_i \Lambda )^{p-2}\leq r_i^ {p-2} \left[1+\frac{1}{4(p-1)} \right]^{p-2}\leq 2r_i^ {p-2}\, .
\end{equation}

Using \eqref{eq:Rit bound} together with \eqref{eq:second deriv p big} and the definition of $\Lambda^*$, we obtain that  
\begin{equation}
\forall \Lambda \in [0,\Lambda^*], \quad     \Bigl(\max_{\lambda\in [0,\Lambda]} \en_p''(\lambda)\Bigr)\Lambda\leq  p\frac{\rho_i}{2} \,,
\end{equation}
so by \eqref{deriv0},  
\begin{equation}
\forall \Lambda \in [0,\Lambda^*], \quad     \en_p'(\Lambda) \le \en_p'(0) +\Bigl(\max_{\lambda\in [0,\Lambda]} \en_p''(\lambda)\Bigr)\Lambda\leq -p\frac{\rho_i}{2} \,.
\end{equation}
Therefore, the decrease of the energy due to replacing $g_i$ with $g_i(\Lambda^*)$ is at least

\begin{equation}
    \Lambda^* p\frac{\rho_i}{2}=\frac{p}{8(p-1)}\frac{
    \rho_i^2}{r_i^{p-2} \delta_i} = \tilde{c}_p I_i \, ,
\end{equation}
where $\tilde{c}_p=\frac{p}{8(p-1)}$.
This finishes the case $p> 2$ and   completes the proof of the claim.

\end{proof}

\begin{proof}[ Proof  of Lemma \ref{le:main p}] 
To bound the ratio \eqref{eq:ratio} of the expected energy decrease to the current energy of $g$, we may 
multiply $g$ by a constant so that $\sum_i \rho_i=1$. For vertices  $i,j$ with $i \sim j$, we write  $$ \Delta_{ij}=|g_i-g_j| \,.$$
Recall the definitions of $R_i^+$ and $R_i^-$ from \eqref{eq:Rplus} and \eqref{eq:Rminus}. For fixed $k \le n$, changing the order of summation gives
\begin{equation} \label{eq:green}
\begin{split}
    \sum_{i\leq k}(R_i^- - R_i^+) &= \sum_{i\leq k} R_i^- -\sum_{j\leq k} R_j^+  \\
     &=\sum_{i\leq k}\sum_{  j>i \atop  j\sim i}\Delta_{ij}^{p-1} - \sum_{j\leq k}\sum_{   i<j \atop i\sim j}\Delta_{ij}^{p-1} = \sum_{i \le k}\sum_{j >k \atop j\sim i} \Delta_{ij}^{p-1} \, .
    \end{split} 
\end{equation}
Indeed, every term $\Delta_{ij}^{p-1}$ such that $i<j \le k$ and $i \sim j$ appears in the penultimate summation twice, with opposite signs.  
The LHS of \eqref{eq:green} is bounded by $\sum_i \rho_i = 1$, so the RHS is at most $1$ as well. In particular, this implies that all summands on the RHS are $\leq 1$, and therefore \begin{equation}\label{eq:allq}
    \sum_{i \le k}\sum_{j >k \atop j\sim i} \Delta_{ij}^{s}\leq 1 \,\ \quad  \forall s\geq p-1 \, .
\end{equation}
  Summing \eqref{eq:allq} over all $k$ yields
\begin{equation}\label{eq:sum e} 
\sum_{i <n}\sum_{j >i \atop j\sim i} (j-i) \Delta_{ij}^{s}\leq n \,\ \quad  \forall s\geq p-1 \, .
\end{equation}
By taking $s=p$, we obtain the following bound on the energy:
\begin{equation}\label{eq:energybound}
    \en_p(g) = \sum_{i <n}\sum_{j >i \atop j\sim i}   \Delta_{ij}^{p}\leq  
    \sum_{i <n}\sum_{j >i \atop j\sim i} (j-i) \Delta_{ij}^{p}\leq n \,.     
\end{equation}

We now proceed to derive a key estimate on the derivative of the energy.\\ For fixed $i$, 
\begin{equation}
    \sum_{j:j\sim i}|j-i|\Delta_{ij}^{p-1} = \sum_{ j>i: j\sim i}(j-i)\Delta_{ij}^{p-1}+ \sum_{ j<i:  j\sim i}(i-j)\Delta_{ij}^{p-1} \, .
\end{equation}
Writing $d_i^+ = \#\{j: j\sim i, j<i\}$ and $d_i^-=\#\{j: j\sim i,j>i\}$, we have
\begin{equation}
\sum_{ j>i: j\sim i}(j-i)\Delta_{ij}^{p-1}
      = d_i^-\left(\frac{1}{d_i^-}\sum_{ j>i: j\sim i}(j-i)\Delta_{ij}^{p-1}\right)  \,.
\end{equation}
To bound the right-hand side, we will apply Chebyshev's   ``other'' inequality (see, e.g., \cite{MR0944909} pp.\ 43)   to the uniform probability measure on $\{j>i: j\sim i\}$, using the fact that both $j-i$ and $\Delta_{ij}=g_i-g_j$  are non-decreasing  in $j$ for $j>i$. Thus
$$ \sum_{ j>i: j\sim i}(j-i)\Delta_{ij}^{p-1} \geq d_i^-\left( \frac{1}{d_i^-}\sum_{ j>i: j\sim i}(j-i) \right) \left(\frac{1}{d_i^-}\sum_{ j>i: j\sim i} \Delta_{ij}^{p-1}\right)\, .
$$
Since  $\displaystyle 
\sum_{j>i \atop j\sim i}(j-i)\geq \frac{(d_i^-)^2}{2}$, we deduce that for each vertex $i$,
\begin{equation}
    \sum_{ j>i: j\sim i}(j-i)\Delta_{ij}^{p-1} \geq \frac{d_i^-}{2}R_i^- \, .
\end{equation}
Similarly, we obtain 
\begin{equation}
\sum_{ j<i: j\sim i}(i-j)\Delta_{ij}^{p-1} \geq    \frac{d_i^+}{2}R_i^+ \, .
\end{equation} 
Since $d_i^+ + d_i^- = d_i$, we combine the two estimates to get 
\begin{equation}
   \forall i, \quad  \sum_{j:j\sim i}|j-i|\Delta_{ij}^{p-1} \geq \frac{d_i}{2}\min(R_i^+,R_i^-) \, .
\end{equation}
Summing over $i$ and using   \eqref{eq:sum e}, we obtain that
\begin{equation}
    \sum_i d_i\min(R_i^+,R_i^-)\leq 4n.
\end{equation}
Since $R_i=R_i^+ + R_i^-=2\min(R_i^+,R_i^-)+\rho_i$, it follows that
\begin{equation}\label{eq:sumridi}
    \sum_i R_id_i\leq 8n+\sum_i\rho_i d_i \leq 9n \, .
\end{equation}

Denote the energy decrease in one update by $\Delta(\en_p):= \en_p(g)-\en_p(g^\#) $. Recall that our goal is to compare  $\E[\Delta(\en_p)]$
to the total energy $\en_p(g)$. 
We will use the bound $\en_p(g)\leq n$ from \eqref{eq:energybound}, while
our lower bound for $\E[\Delta(\en_p)]$ will depend on $p$.
By Claim \ref{cl:p improve}, we have 
\begin{equation}\label{eq:lower delta energy} \E[\Delta(\en_p)] \geq \frac{\tilde{c}_p}{n} \sum_i I_i \,,
\end{equation}
where $I_i$ were defined in \eqref{eq:ii}. We once again consider several cases.\\ 
\smallskip

\textbf{Case} $1<p\leq 2$: 

In this case,
\begin{equation}
     \sum_i I_i= \sum_i \rho_i^\frac{p}{p-1} d_i^\frac{-1}{p-1}\geq  \frac{\left(\sum_i \rho_i\right)^\frac{p}{p-1}}{\left(\sum_id_i\right)^\frac{1}{p-1}} \, ,
\end{equation}
where the inequality comes from applying \hld to $d_i^\frac{1}{p}$  and $\rho_i d_i^{-\frac{1}{p}}$,   with exponents $p$ and $q=\frac{p}{p-1}$.

Using our normalization that $\sum_i \rho_i=1$ and writing $D$ for the average degree in $G$, we get
\begin{equation}
\E[\Delta(\en_p)] \geq \frac{\tilde{c}_p}{n}\sum_i I_i \geq \tilde{c}_p n^{-\frac{p}{p-1}}D^{-\frac{1}{p-1}} \, . 
\end{equation}
Using \eqref{eq:energybound}, we conclude that
\begin{equation}\label{eq:p between one and two}
\E\Bigl[\frac{\Delta (\en_p)}{\en_p(g)}\Bigr]\geq \tilde{c}_p n^\frac{1-2p}{p-1}D^\frac{-1}{p-1}=c_p F(n,p,D)\, .
\end{equation}  \\
\textbf{Case}  $p> 2$:\\
By \eqref{eq:ii},
\begin{equation}\label{eq:deltabound}
     \sum_i I_i =  \sum_i {\rho_i^2}r_i^{2-p}\delta_i^{-1}  \geq  
     \frac{(\sum_i \rho_i)^2}{\sum_i r_i^{p-2} \delta_i} = \frac{1}{\sum_i r_i^{p-2} \delta_i}  \, ,
\end{equation}
where the inequality follows from applying Cauchy-Schwarz to the two sequences $ \rho_i   \bigl(r_i^{2-p} \delta_i^{-1} \bigr)^{1/2}  $ and 
$\bigl(r_i^{p-2} \delta_i  \bigr)^{1/2}  $.

When $p\geq 3$, we have
\begin{equation}
   \sum_i r_i^{p-2} \delta_i
   =\sum_i R_i^\frac{p-2}{p-1}d_i^\frac{1}{p-1}
   \leq \sum_i (R_id_i)^\frac{p-2}{p-1} \leq 10n \, ,
\end{equation}
where the last inequality followed from \eqref{eq:sumridi}, together with the fact that for each $i$, either $0\leq R_id_i<1$ or $(R_id_i)^\frac{p-2}{p-1}\leq R_id_i$.

Combining the last inequality  with \eqref{eq:energybound},   \eqref{eq:lower delta energy} and \eqref{eq:deltabound}, we conclude that
\begin{equation}\label{eq:p bigger three}
   \E\Bigl[\frac{\Delta (\en_p)}{\en_p(g)}\Bigr]\geq \frac{\tilde{c}_p}{10}n^{-3} =c_p F(n,p,D) \, .
\end{equation}

The last subcase remaining is $2\leq p<3$. In this range,
applying \hld with $\widetilde{p}=\frac{p-1}{p-2}$ and $\widetilde{q}=p-1$, we obtain
\begin{equation} 
\begin{split}
    \sum_i r_i^{p-2} \delta_i 
    &= \sum_i (R_i d_i)^\frac{p-2}{p-1}d_i^\frac{3-p}{p-1} \le
\Bigl(\sum_i R_id_i\Bigr)^{\frac{p-2}{p-1}}\Bigl(\sum_id_i^{3-p}\Bigr)^\frac{1}{p-1} \\&\leq  (9n)^\frac{p-2}{p-1}n^{\frac1{p-1}}\Bigl(\frac1n\sum_i d_i^{3-p}\Bigr)^\frac{1}{p-1} \leq 9n  D^\frac{3-p}{p-1} \, .
\end{split}
\label{holder-jensen}
\end{equation}
 The penultimate inequality follows from \eqref{eq:sumridi}, and the final one follows from Jensen's inequality. Substituting the above bound into \eqref{eq:deltabound}, we obtain that 
\begin{equation}\label{eq:p between two and three} 
   \E\Bigl[\frac{\Delta (\en_p)}{\en_p(g)}\Bigr]\geq \frac{\tilde{c}_p}{9}n^{-3}D^\frac{p-3}{p-1}>c_pF(n,p,D) \, .
\end{equation}



Combining the three cases \eqref{eq:p between one and two}, \eqref{eq:p bigger three}, and \eqref{eq:p between two and three}, we get that for all $p>1$,
\begin{equation}
    \E\Bigl[\frac{\Delta (\en_p)}{\en_p(g)}\Bigr]\geq c_p F(n,p,D) \, .
\end{equation}
This completes the proof of Lemma \ref{le:main p}.
\end{proof}

{\bf{Proof of Theorem \ref{th:main lpbound} (a) and (b)}.}

\begin{proof}
Part (a) follows from Lemma \ref{le:main p} by induction. 
To prove part (b), suppose that $\osc(f_t)>\epsilon$. Then there is some edge $e$ with $\nabla f_t(e)\geq \epsilon/n$, and this implies that $\en_p(f_t)\geq (\epsilon/n)^p$.
On the other hand, since $\en_p(f_0) \le n^2$,  part (a) gives 
$$\E[\en_p(f_t)] \le n^2 \exp\bigl(-c_p F(n,p,D)t\bigr)\,,$$
so by Markov's inequality, 
$$\Pr\bigl(\en_p(f_t) >(\epsilon/n)^p \bigr) \le \frac{n^2 \exp\bigl(-c_p F(n,p,D)t \bigr)} {(\epsilon/n)^p}  \,. $$
The RHS is less than 1 if 
$$t>t_\epsilon:=\frac{\log(n^{2+p}\epsilon^{-p})}{c_pF(n,p,D)} \,.$$
Therefore,
\begin{equation}
\begin{split}
\E[\tau_p(\epsilon)] &=
\int_0^\infty \!\!\! \Pr(\osc(f_{\lfloor t \rfloor})>\epsilon) \,dt \le 
t_\epsilon+1+\int_{t_\epsilon+1}^\infty \!\! \!\!\exp\bigl(-c_p F(n,p,D)(t-1-t_\epsilon)\bigr) dt\\[1ex]
&=t_\epsilon+1+\frac1{c_p F(n,p,D)} 
\le \frac{(2+p)\log(n/\epsilon)}{c_pF(n,p,D)}
\,,
\end{split}   
\end{equation}
provided that $\epsilon \le 1/2$.

\end{proof}

\section{Lower bounds}\label{s:lowerbounds}
\subsection{The energy minimizing dynamics on the cycle}

In this subsection we analyze the $\ell^p$-energy minimizing dynamics on the cycle with arbitrary update sequence.
Note that on the cycle, the $\ell^p$-energy minimizing dynamics are the same for all $p>1$; in each step, the opinion of the chosen vertex is replaced by the average of the opinions of its two neighbours. The following theorem gives a lower bound for the consensus time on the cycle  which   shows that the dependence on $n$ in our upper bounds is tight for $p\geq 3$  (including $p=\infty$) up to a $\log n$ factor. 

\begin{thm} \label{thm:cycle1}
    Suppose that $4|n$. Consider the cycle $\mathrm{C}_n =\{0,1,2,\ldots,n-1\}$  with the initial profile $f_0\equiv 1_{v\geq n/2}$. Then for every update sequence $\{v_s\}_{s\geq 1}$ and every $t\leq \frac{n^3}{2048}$, we have $\osc(f_t)\geq \frac{1}{2}$.
\end{thm}
\begin{proof}
    Let $\{v_s\}_{s\geq 1}$ be an arbitrary update sequence. 

To analyze $f_t(w)$ we consider the following fragmentation process, which can be thought of as a dual of the averaging process. We initialize the process with $\mu_0=\delta_w$, and $\mu_{k+1}$ is obtained from $\mu_k$ by splitting the mass at $w_k$ equally among its neighbours. We let $w_k=v_{t-k}$ for $0\leq k \leq t-1$.

Using induction on $j$, we can verify that $f_t(w)= \sum_{v}\mu_j(v)f_{t-j}(v)$ for each
$j=0,\dots,t$. In particular, $$f_t(w)=\sum_v \mu_t(v)f_0(v).$$
We now apply the following result from \cite{paterson2009maximum}:
\begin{prop}\label{p:mass bound}
    If $\mu_t(B(w,r)^c)=\theta$, then $t\geq \theta^2 r^3/2$, where $B(w,r)$ is the open ball of radius $r$ in the graph distance. 
\end{prop}
For convenience we recall the short proof of the proposition.
\begin{proof}
    It suffices to consider the fragmentation process on $\Z$, with $w=0$. If some fragmentation steps were performed outside of $B(0,r)$, then removing the last of these steps will yield a new update sequence $\{\tilde{v}_s\}_{s\geq 1}$ for which $\tilde{\mu}_{t-1}(B(w,r)^c) \ge \theta$. 
 Iterating, we may assume that no fragmentation steps were performed outside of $B(0,r)$.
     Let $$Q(\mu):=\sum_{i\in \Z} i^2 \mu(i) \,\  \text{and} \,\  \en(\mu):=\sum_{i,j} |i-j|\mu(i)\mu(j).$$
     (Both sums converge as there are only finitely many non-zero terms).
    Then it is easily verified that
    $$Q(\mu_{k+1})-Q(\mu_k)=h_k \, \text{ and } \, \en(\mu_{k+1})-\en(\mu_k)=h_k^2 \, ,$$ where $h_k=\mu_k(w_k)$. 
    If $\theta = \mu_t(B(0,r)^c)$, then, using Cauchy-Schwarz, we get 
    \begin{equation*}
        \left(\theta r^2\right)^2 \leq \left(Q(\mu_t)-Q(\mu_0)\right)^2 = \left(\sum_{k=0}^{t-1}h_k\right)^2\leq t\sum_{k=0}^{t-1}h_k^2 \\
        = t\left(\en(\mu_t)-\en(\mu_0)\right) \leq 2rt.
    \end{equation*}
    In the last inequality, we used the fact that, by our assumption on the updates, no mass travelled beyond $[-r,r]$. 
    We conclude that $t\geq \theta^2 r^3/2$.
\end{proof}
We can now finish the proof of Theorem \ref{thm:cycle1}.
Using the Lemma, we deduce that for $t\leq (n/4)^3/32$, we have $\mu_t(B(\frac{n}{4},\frac{n}{4})^c)\leq \frac{1}{4}$, so $f_t(\frac{n}{4})\leq \frac{1}{4}$. Similarly $f_t(\frac{3n}{4})\geq \frac{3}{4}$, whence $\osc(f_t)\geq \frac{1}{2}$ for such $t$. 

    \end{proof}

\begin{figure}[h]  
    \centering
\begin{tikzpicture}
    \node[circle, fill, red, inner sep=1pt, label={[text=red]left:$0$}] (A) at (-3,0) {};
    \node[circle, fill, blue, inner sep=1pt, label={[text=blue]right:$1$}] (B) at (3,0) {};
    \node[draw=none, fill=none] (C) at (0,0.1) {$\vdots$};

    \node[inner sep=1pt, label={[text=red]left:$0$}] at (-1.8,0) {};
    \node[inner sep=1pt, label={[text=red]left:$0$}] at (-0.9,0) {};
    \node[inner sep=1pt, label={[text=blue]left:$1$}] at (2.3,0) {};
    \node[inner sep=1pt, label={[text=blue]left:$1$}] at (1.4,0) {};

    \foreach \i in {-3,-2,-1,1,2,3} {
        \draw[marks,nobreak,middots] (A) to[out=18*\i, in=180-18*\i] (B);
    }
\end{tikzpicture}

\caption{A graph consisting of $k$ parallel paths of length $L$ between two nodes, where $4|L$. This graph, with initial profile 0 on the left half and 1 on the right half, yields the lower bound in Theorem \ref{th:linfty no boundary simplified}.}
    \label{fig:lower necklace2}
\end{figure}
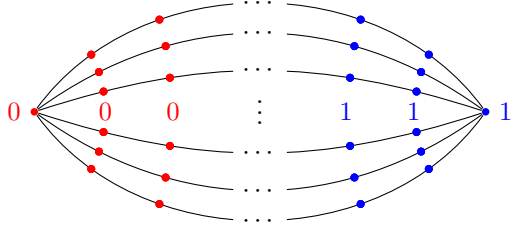 

The techniques used to analyze the dynamics on the cycle also yield tight diameter dependent lower bounds for Lipschitz learning:
\begin{cor} Suppose that $4|L$. 
    Consider the graph $G=(V,E)$ in Figure \ref{fig:lower necklace2} consisting of $k \ge 2$ paths of length $L$ that connect two vertices $a$ and $z$, so $n=|V|=k(L-1)+2$. Let $f_0$ take value $0$
on the left half of $G$ and $1$ on the right half. Then for any update sequence, we have $\tau_\infty(1/2) \ge cnL^2$ with $c= 2^{-11}$.
\end{cor}

\begin{proof}
Fix  $t<cnL^2$. There is a simple 
path $\gamma=\{\gamma_0,\dots,\gamma_L\}$ connecting $a$ to $z$, so that among the first $t$ updates,  at most $t/k \le cL^3$ of them took place in $\{\gamma_1,\dots,\gamma_{L-1}\}$. Let $v:=\gamma_{L/4}$ be a vertex in the middle of the left half of $\gamma$, so $f_0$ vanishes on the ball $B(v,L/4)$ in $\gamma$.  
As in the first proof for the cycle, Proposition \ref{p:mass bound} implies that $f_t(v)< 1/4$ if 
$t/k<2^{-11}L^3$. Similarly, the vertex $w:=\gamma_{3L/4}$ in the right half of $\gamma$ satisfies $f_0(w) \ge 3/4$ for such $t$, so $\osc(f_t) >1/2. $
Thus $\tau_\infty(1/2) \ge cnL^2$ with $c= 2^{-11}$.
\end{proof}

\subsection{A second approach to the lower bound on the cycle}
In this subsection we give a second proof of the lower bound for the averaging dynamics on the cycle. This proof, while longer,  will be useful for obtaining lower bounds on more complicated graphs.
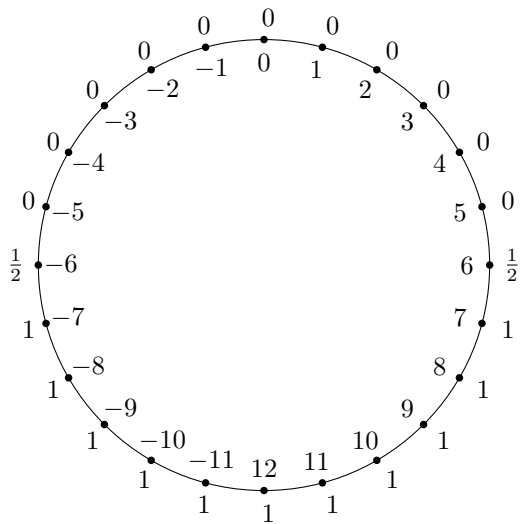
\begin{figure}
    \centering
    \begin{tikzpicture}
  \def\n{6} 
  \def\radius{3cm} 

  \draw (0,0) circle[radius=\radius];

  \foreach \i in {0,...,\numexpr4*\n-1} {
    \pgfmathtruncatemacro{\ii}{\i-2*\n+1};
    \pgfmathsetmacro{\angle}{90-\ii*360/(4*\n)};
    \pgfmathtruncatemacro{\label}{\ii};

    \node at (\angle:\radius) [circle, fill, inner sep=1pt] {};
    \node at (\angle:\radius - 0.3cm) {$\label$};
    \def\val {\ifthenelse{\ii=\n \OR \ii=-\n}{$\frac{1}{2}$}{ \ifthenelse{\ii > \n \OR \ii < -\n}{$1$}{$0$} }};
    \node at (\angle:\radius + 0.3cm) {\val};
  }
\end{tikzpicture}
    \caption{The cycle graph $C_{4n}$ for $n=6$ and the initial profile $f_0$ used in Theorem \ref{secondcycle}.}
    \label{fig:enter-label}
\end{figure}

\begin{thm} \label{secondcycle}
    Consider the cycle $C_{4n} = \{-2n+1,-2n+2,...,-1,0,1,\ldots,2n\}$ with initial profile $f_0(v):= \frac{1}{2}\mathfrak{1}_{|v|=n} + \mathfrak{1}_{|v|>n}$. Then for every update sequence $\{v_s\}_{s\geq 1}$ and every $t<\frac{n^3}{68}$, we have $\osc(f_t)\geq \frac{1}{2}$.
\end{thm}

The main step in the proof of Theorem \ref{secondcycle} is the following claim, proved below.

   \begin{claim}\label{cl:segment}
    Consider the averaging dynamics $\{F_t(v)\}$ on the line segment $S=\{-n,-n+1,\ldots,0,1,\ldots,n\}$, with the initial profile $F_0(v)=\frac{v^2}{n^2}$. Then  for every update sequence  $\{v_t\}$ that does not include the endpoints $\pm n$, and every $t<\frac{n^3}{68}$, we have $F_t(0)\leq \frac{1}{4}$.
    \end{claim}

\begin{proof}[ Proof of Theorem \ref{secondcycle}]
By symmetry,  it is enough to prove that \begin{equation} \label{eq:proof68} f_t(0)\leq \frac{1}{4} \quad \text{for every}  \quad t<\frac{n^3}{68}\,.
\end{equation}
Extend the function $F_0$, defined in  Claim \ref{cl:segment}, to the cycle $C_{4n}$, by setting $F_0(v)=1$ for $v \in C_{4n}\setminus S$, and observe that $f_0 \le F_0$ on $C_{4n}$.   Update  $f_t$ and  $F_t$ using the same update locations, except that when $f_t$ undergoes an update at $\pm n$,  it is skipped for  $F_t$.
 The monotonicity property of the dynamics and the inequality $f_t \le 1$ yield by induction that $f_t \le F_t$ for all $t$.  Thus \eqref{eq:proof68}  follows from Claim \ref{cl:segment}.
\end{proof}

\begin{figure}
    \centering
    \begin{tikzpicture}[inner sep=0pt, outer sep = 0pt]
        \tikzstyle{every node}=[font=\large]
        \node [label=below:{$(-1,\phi(-1))$}] (A) at (0,0) {};
        \node[label=above:{$(1,\phi(1))$}] (B) at (8,6) {};
        \node (M) at (4,3) {};
        \node (C) at (4,4) {};
        \node (F) at (2.1,2.1) {};
        \node (G) at (2.2,2.2) {};
        \node (U) at (4,5) {};
        \draw [densely dashed] (A) -- (B);
        \draw (F) -- node[above=2pt] {$\phi$} ++(C);
        \draw (A) -- (G);
        \draw (B) -- (C);
        \draw (C) -- node[right=2pt] {$\gamma$} ++(M);
        \draw (C) -- node[right=2pt] {$\beta$} ++(U);
        \draw [dashed] (A) -- (U);
        \draw [dashed] (B) -- (U);
    \end{tikzpicture}
    \caption{The increase in length of the graph of the function $\phi$ when adding $\beta$ to the value at $0$ is bounded by $4\beta\max(\beta,\gamma)$ \,.}   \label{fig:geo}
\end{figure}
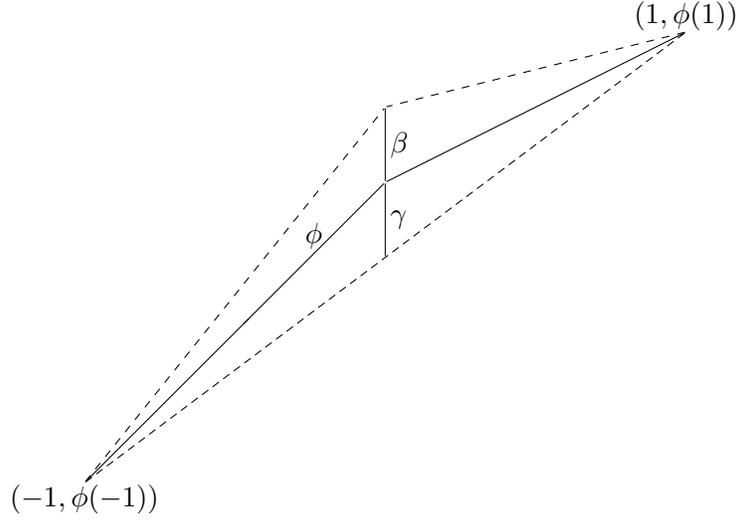

To prove Claim \ref{cl:segment} we will need the following elementary geometric fact (see Figure \ref{fig:geo}):
    \begin{lem}\label{le: geometric lemma}
   Let $\phi:[-1,1]\rightarrow \R$ be linear in each of the intervals $[-1,0]$ and   $[0,1]$. Suppose  that $\phi$   satisfies
   $$0 \le \gamma:= -\Delta \phi(0)=\phi(0)-\frac{ \phi(-1)+ \phi(1)}{2} \le \alpha\,.$$
   Given $\beta \in [0,\alpha]$, define the function $\phi_\beta:[-1,1]\rightarrow \R$ by setting $\phi_\beta(\pm 1)=\phi(\pm 1)$ and $\phi_\beta(0)=\phi(0)+\beta$, with linear interpolation in between. Then $$\ell(\phi_\beta)\leq \ell(\phi)+4\alpha\beta\,,$$
   where $\ell(\phi)$ is the length of the graph of $\phi$.
    \end{lem}
    \begin{proof}
        Let $m=\frac{\phi(1)-\phi(-1)}{2}$.
        Writing $f(x)=\sqrt{(m+x)^2+1}+\sqrt{(m-x)^2+1}$, we get
        \begin{equation}\label{eq:length difference one step}
            \ell(\phi_\beta)-\ell(\phi) = f(\gamma+\beta)  -f(\gamma) \leq  \beta \sup_{\gamma\leq x\leq \gamma+\beta} {f'(x)} \, .
        \end{equation}
        Let $g(x)= \frac{m+x}{\sqrt{(m+x)^2+1}}$. We can write the derivative of $f$ as 
        \begin{equation}\label{eq:f derivative bound}
            f'(x)=\frac{m+x}{\sqrt{(m+x)^2+1}} - \frac{m-x}{\sqrt{(m-x)^2+1}} = g(x)-g(-x) \, .
        \end{equation}
        To bound the above difference, we once again take the derivative:
        \begin{equation}\label{eq:g derivative}
            g'(x)=\frac{1}{\Bigl((m+x)^2+1\Bigr)^{\frac{1}{2}}}-\frac{(m+x)^2}{\Bigl((m+x)^2+1\Bigr)^{\frac{3}{2}}}=\frac{1}{\Bigl((m+x)^2+1\Bigr)^{\frac{3}{2}}} \leq 1.
        \end{equation}
        Thus $g(x)-g(-x) \le 2x$ for $x>0$. This, together with \eqref{eq:f derivative bound}, implies that $$\sup_{\gamma\leq x\leq \gamma+\beta} {f'(x)}\leq 2(\gamma+\beta) \leq 4\alpha\,.$$
        The lemma now follows from \eqref{eq:length difference one step}.
        
    \end{proof}
 
\begin{proof}[Proof of Claim \ref{cl:segment}]
    First note that $F_0$ is convex and therefore $F_t$ is convex for all $t$. Also, $F_0$ was chosen so that for every $|x|<n$ , its discrete Laplacian satisfies
    \begin{equation}\label{eq:w_0 laplacian}
        \Delta F_0(v) := \frac{F_0(v+1)-2F_0(v)+F_0(v-1)}{2}=\frac{1}{n^2}.
    \end{equation}
    Let $g_t(v):=F_t(v)-F_0(v) = F_t(v)-\frac{v^2}{n^2}$. Note that $g_0\equiv 0$ and that $g_t(\pm n)=0$ for all $t$. A simple calculation now gives that for all $t \ge 1$, 
    \begin{equation}\label{eq:g increment}
    \begin{split}
        g_t(v_t)&=F_t(v_t)-F_0(v_t)= \frac{F_{t-1}(v_t-1)+F_{t-1}(v_t+1)}{2}-F_0(v_t)\\ &= \frac{g_{t-1}(v_t-1)+g_{t-1}(v_t+1)}{2}+\frac{F_0(v_t-1)+F_0(v_t+1)}{2}-F_0(v_t) \\ &= \frac{g_{t-1}(v_t-1)+g_{t-1}(v_t+1)}{2}+\frac{1}{n^2}\,.
        \end{split}
    \end{equation}

    By convexity of $F_t$ and \eqref{eq:w_0 laplacian}, for every $|v|<n$ and every $t$, we have
    \begin{equation}\label{eq:Delta g}
        \Delta g_t(v)= \Delta F_t(v)-\Delta F_0(v) \geq \frac{-1}{n^2} \,.
    \end{equation}

    We now take $h_t$ to be the concave envelope of $g_t$, that is, 
    \begin{equation*} \label{concenv}
      \forall t \ge 0, \quad   h_t(v):= \inf \{ h(v) \, : \, \, h \geq g_t \, \text{and} \, h \, \text{is concave}\} \,.
    \end{equation*}

    Note that for each $t$, the envelope $h_t$ coincides with $g_t$ on some subset of vertices  $A_t\subset \{-n,\ldots, n\}$ and linearly interpolates between these values. 
    Since $h_t\geq g_t$, we deduce that  $\Delta h_t(v)\geq \Delta g_t(v)$ for $v\in A_t$ and $\Delta h_t(v)=0$ elsewhere.
    The concavity of $h_t$ and \eqref{eq:Delta g} give that for all $t \ge 0$,  
    \begin{equation}\label{eq:Delta h}
        0\geq \Delta h_t \geq \frac{-1}{n^2} \,.
    \end{equation}
   
    Moreover, for all $t \ge 1$,  \begin{equation*}
        h_{t-1}(v_t)\geq \frac{g_{t-1}(v_t-1)+g_{t-1}(v_t+1)}{2}\, ,
    \end{equation*}
    so \eqref{eq:g increment} implies that
    \begin{equation}
        g_{t}(v_t)\leq h_{t-1}(v_t)+\frac{1}{n^2}\,.
    \end{equation}
Since $g_{t}(v) = g_{t-1}(v)\leq h_{t-1}(v)$ for all $v\neq v_t$, we get that
        $g_{t} \leq h_{t-1}+\frac{1}{n^2}$.
Therefore, \begin{equation}\label{eq: h increment}
    h_{t}\leq h_{t-1}+\frac{1}{n^2} \,.
\end{equation}
    Extend the function $h_t$ to a function from $[-n,n]$ to $[0,1]$ by linear interpolation. As before, let $\ell(h_t)$ denote the length of the graph of this function.  Recall that $h_0\equiv 0$, and therefore $\ell(h_0)=2n$. 
\begin{figure}
    \centering
\begin{tikzpicture}[x=1cm, y=0.4cm]
    \node[inner sep = 0mm, outer sep = 0mm] (A1) at (0,0) {};
    \node[inner sep = 0mm, outer sep = 0mm] (A2) at (1.5,3.5) {};
    \node[inner sep = 0mm, outer sep = 0mm] (A3) at (3,4.3) {};
    \node[inner sep = 0mm, outer sep = 0mm] (A4) at (4.5,4.8) {};
    \node[inner sep = 0mm, outer sep = 0mm] (B4) at (4.5,6.6) {};
    \node[inner sep = 0mm, outer sep = 0mm] (A5) at (6,3.6) {};
    \node[inner sep = 0mm, outer sep = 0mm] (A6) at (7.5,2) {};
    \node[inner sep = 0mm, outer sep = 0mm] (A7) at (8.5,0) {};
    
    \draw (A1) -- (A2);
    \draw (A2) -- (A3);
    \draw (A3) -- (A4);
    \draw (A4) -- (A5);
    \draw (A5) -- (A6);
    \draw (A6) -- (A7);
    \draw[very thick,blue, dashed] (A2) -- (B4);
    \draw[very thick,blue, dashed] (A6) -- (B4);
    \draw[very thick,red, densely dotted] (A3) -- (B4);
    \draw[very thick,red, densely dotted] (A5) -- (B4);
    \draw[->, very thick] (A4) -- (B4);
     \node at (A2) [below  ] {$a$};
    \node at (A6) [below ] {$b$};
    \node at (A4) [below ] {$v_t$};
\end{tikzpicture}
    \caption{Bounding the increase in length $\ell(h_{t+1})-\ell(h_t)$ using the triangle inequality
    and Lemma  \ref{le: geometric lemma} }\label{fig:triangle inequality}.  
\end{figure}
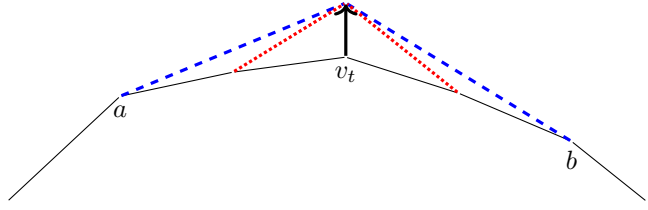

Consider the update done at time $t \ge 1$ at  $v_t$. 
Next, we will show that  
\begin{equation} \label{len-inc}
\ell(h_{t})-\ell(h_{t-1})\leq 4n^{-4} \,.
\end{equation}
    We may assume that $\beta:=g_{t}(v_t)-h_{t-1}(v_t)$ is positive, since otherwise $h_{t}=h_{t-1}$ on $[-n,n]$. 
    Set $\phi(x):=h_{t-1}(x+v_t)$ for $x \in [-1,1]$.
    The inequalities \eqref{eq:Delta h} and  \eqref{eq: h increment} ensure that the conditions of Lemma \ref{le: geometric lemma} are satisfied with $\beta \le \alpha:=\frac{1}{n^2}$.
    
    There exists $a<v_t$ such that $h_{t}$ is linear on the interval $[a,v_t]$
    and $$h_{t}(a)=g_{t}(a)=g_{t-1}(a)\,.$$ Similarly, there exists $b>v_t$ such that $h_{t}$ is linear on the interval $[v_t,b]$ and satisfies 
    $h_{t}(b)=g_{t}(b)=g_{t-1}(b)$. Thus, $h_{t}=h_{t-1}$ on 
    $[-n,a] \cup [b,n]$. By the triangle inequality (see Figure \ref{fig:triangle inequality}),
    \begin{equation*}
       \ell(h_{t}|_{[a,b]})  \leq   
    \ell(h_{t-1}|_{[a,v_t-1]})+\ell(\phi_\beta)+\ell(h_{t-1}|_{[v_t+1,b]}) \,.
    \end{equation*}
    Therefore, by Lemma \ref{le: geometric lemma}, 
    $$
\ell(h_{t}|_{[a,b]})-\ell(h_{t-1}|_{[a,b]}) \leq   
      \ell(\phi_\beta)-\ell(\phi) \le 4\alpha\beta \le 4n^{-4}\,,
      $$
  and we have verified \eqref{len-inc}.

    If $f_T(0)\geq \frac{1}{4}$ for some $T$, then
     $h_T(0)\geq \frac{1}{4}$, so
     \begin{equation}
        \ell(h_T)\geq 2\sqrt{n^2+(1/4)^2}\geq 2n+\frac{1}{17n}.
    \end{equation}
     Since $\ell(h_0)=2n$, we infer that $T\geq \frac{n^3}{68}$, and Claim \ref{cl:segment} follows.
\end{proof}

\subsection{Lower bound for $1<p\leq 2$ depending on average degree.} \label{s:p12lower}
We will prove the converse statement in Theorem \ref{th:main lpbound} in two parts. First, for each large $N$ we will exhibit a tree with at most $N$ vertices  (so the average degree is less than 2) for which the lower bound \eqref{eq:p D lower} holds.  Second, for some constant $D(p)$ and every $D >D(p)$, we will construct for each large $N$  a graph  with at most $N$ vertices and average degree at most $D$, for which \eqref{eq:p D lower} holds.

\begin{figure}[h]
\centering
    \begin{tikzpicture}
    \filldraw[black] (0,0) circle (0.05cm) node[anchor=north] {$0$};
    \foreach \s in {-1, 1} {
        \draw[thick] (0,0) -- (\s,0);
        \draw (\s*1.3, 0) node {$\cdots$};
        \draw[thick] (\s*1.6,0) -- (\s*2.6,0);
        \ifthenelse{\s=1}{\filldraw[black] (\s*2.6,0) circle (0.05cm) node[anchor=north east] {$n$}}{\filldraw[black] (\s*2.6,0) circle (0.05cm) node[anchor=north west] {$-n$}};
    }

    \foreach \sc in {-1, 1}
    {
        \foreach \s in {-1, 1}
        {
            \foreach \i in {1, ..., 4}
            {
                \pgfmathtruncatemacro\j{5-\i};
                \pgfmathtruncatemacro\jj{4-\i};
                \ifthenelse{\sc=-1}{\def\mylabel{\ifthenelse{\s=1}{$w_{\j}$}{\ifthenelse{\jj=0}{$w_{2n}$}{$w_{2n-\jj}$}}}}{\def\mylabel{\ifthenelse{\s=1}{$u_{\j}$}{\ifthenelse{\jj=0}{$u_{2n}$}{$u_{2n-\jj}$}}}}
                
                \ifthenelse{\sc=1}{\filldraw[black] (3.6*\sc, \s*\i/2) circle (0.05cm) node[anchor=west] {\mylabel} ;}{\filldraw[black] (3.6*\sc, \s*\i/2) circle (0.05cm) node[anchor=east] {\mylabel} ;}
                \draw[thick] (\sc*2.6, 0) -- (3.6*\sc, \s*\i/2);
            }
        }
        \draw (3.6*\sc, 0) node {$\vdots$};
    }
\end{tikzpicture}

\caption{The graph $T_n$, consisting of a segment of length $2n$ connected at each endpoint to $2n$ leaves.}
    \label{fig:tree example}
\end{figure}
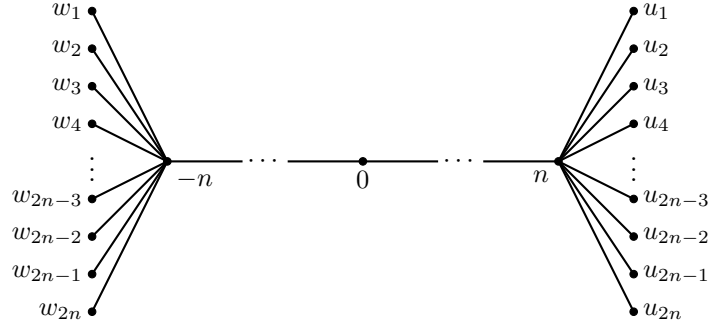

Given a large $N$, let $n=\lfloor (N-1)/6\rfloor$  and consider a tree $T_n$ on $6n+1$ vertices, consisting of a segment of length $2n$, connected at each endpoint to $2n$ leaves, see Figure \ref{fig:tree example}. More precisely, let $T_n=(V,E)$ with
    \begin{eqnarray*}
    V&=&\{w_i\}_{1\leq i\leq 2n} \cup \{i\}_{-n\leq i\leq n}\cup \{u_i\}_{1\leq i\leq 2n}, \\ 
    E&=& \{{w_i,-n}\}_{1\leq i\leq 2n} \cup \{{i,i+ 1}\}_{-n\leq i<n} \cup \{{u_i,n}\}_{1\leq i\leq 2n} \,.
\end{eqnarray*}
The initial profile $f_0$ we choose on $T_n$ is: 
\begin{eqnarray*}
    &&f_0(w_i)= 1 - f_0(u_i) = 0 \quad 1\leq i\leq 2n\, ,\\
    &&f_0(-n)= 1- f_0(n) = n^\frac{-p}{p-1} \, ,\\
    &&f_0(i)= n^\frac{-p}{p-1} + (i+n)\frac{1-2n^\frac{-p}{p-1}}{2n} \quad -n<i<n \, .
\end{eqnarray*}

    We also define $h_0:V\rightarrow [0,1]$ by 
    \begin{eqnarray*}
        &&h_0(w_i)= 1- h_0(u_i) = 0 \quad 1\leq i\leq 2n\, ,\\
        &&h_0(-n)= f_0(-n)=n^\frac{-p}{p-1} \, ,\\
        &&h_0(i) = n^\frac{-p}{p-1} + (i+n)\frac{1-n^\frac{-p}{p-1}}{n} \quad -n<i\leq 0 \, , \\
        &&h_0(i) =1 \quad 1\leq i \leq n \, .
    \end{eqnarray*}

    We are now ready to state our Theorem.
    \begin{thm}
        Run the $\ell^p$-energy minimizing dynamics on the tree $T_n=(V,E)$ defined above with initial opinion profile $f_0$. Then for every large enough $n$, every updating sequence and every $t<\frac{1}{4} n^{\frac{2p-1}{p-1}}$, we have $\osc(f_t)\geq \frac{1}{2}$.
    \end{thm}
    \begin{rem}
        The theorem is valid for all $p>1$ but only yields a sharp bound when $1<p\leq 2$.
    \end{rem}
\begin{proof}
Fix some update sequence. By symmetry, it is enough to prove that for every $t$, we have $$ \min_{i \le 2n} f_t(w_i)  \leq \frac{t}{n} \cdot n^\frac{-p}{p-1} \,.$$ 
Next, we note that $f_0\leq h_0$ and therefore $f_t\leq h_t$ (where $h_t$ is the evolution of $h_0$ under the $\ell^p$-energy minimizing dynamics with the same update sequence). Thus it suffices to prove that for every $t\geq 0$, we have \begin{equation}\label{eq:h2 grows slow}
   \min_{i \le 2n} h_t(w_i) \leq  \frac{t}{n} \cdot n^\frac{-p}{p-1} \,.
\end{equation}
We first verify by induction on $t$  that for all  $t\le n$,
\begin{equation}\label{h2slow}
 \forall j \in [0,n], \quad  h_t(-j) \leq h_0(-j)   
\end{equation}
and 
\begin{equation}\label{h2slow2}
 \forall i \in [1,2n], \quad  h_t(w_i) \leq h_0(-n)  \,.
\end{equation}
The base case is clear. For the induction step, suppose $t+1 \leq n$. The induction step for \eqref{h2slow2} is immediate, since $-n$ is the only neighbour of   $w_i$ for each $i$.    The linearity of $h_0$ on $\{-n,-n+1,\ldots,0\}$ and the induction hypothesis imply that  $h_{t+1}(-j) \leq h_0(-j)$ holds for $j<n$.   

 It remains to prove that  $h_{t+1}(-n) \le h_0(-n)$, or equivalently, the function obtained from $h_t$ by replacing the value at $-n$ with $h_0(-n)$ is $p$-superharmonic at $-n$. The induction hypothesis yields that $h_t(1-n) \le h_0(1-n)$, so  \eqref{eq:psuper} implies that it is enough to check that \begin{equation}
        \sum_{i=0}^{2n} \left(h_0(-n)-h_t(w_i)\right)^{p-1} \stackrel{?}{\geq} \left(h_0(-n+1) - h_0(-n)\right)^{p-1} \, .
    \end{equation}
Since $t \le n$, at least $n$ of the opinions $\{h_t(w_i)\}_{i=1}^{2n}$ equal $0$, so the inequality $h_0(-n+1) - h_0(-n)<\frac{1}{n}$ implies that it suffices to show
 \begin{equation}
        nh_0(-n)^{p-1}\stackrel{?}{\geq} \\ 
 \Bigl(\frac{1}{n}\Bigr)^{p-1} \,.
\end{equation}
This, indeed,  holds for our choice of $h_0(-n)=n^\frac{-p}{p-1}$, completing the induction step. \\
The inequalities \eqref{h2slow} and \eqref{h2slow2} directly imply $h_n \leq h_0 + h_0(-n)$ (pointwise). Another induction (using the  monotonicity of the dynamics and the fact that adding a constant to all opinions is preserved by the dynamics) gives that 
\[h_{kn}\leq h_0+kh_0(-n) \,.\]  

If $kn \le t <(k+1)n$, then the preceding inequality yields that
$$\forall i \le 2n, \quad h_{kn}(w_i) \le k n^\frac{-p}{p-1} \,.$$
Since at most $n$ updates take place in the time interval $[kn,t]$, 
the inequality \eqref{eq:h2 grows slow} follows, completing the proof.


\end{proof}

We now move to the large $D$ case. 
Let $d=\lfloor D-1 \rfloor$. Given a large $N$, we let $n$ be the largest multiple of $d$ such that $4n+1 \le N$.
 Our graph $H_{d,n}$ will consist of a line segment $[-n,n]$, connected at each endpoint to   $m=\frac{n}{d}$ cliques, each of size $d$. The vertex $-n$ is connected to all vertices of $m$ cliques $\{W_i\}_{i=1}^m$, while 
$n$ is connected to all vertices of the $m$ cliques $\{U_i\}_{i=1}^m$, see Figure \ref{fig:Hdn}.
  Thus the maximal degree in this construction is $n+1$, but the average degree is $\leq D$. We write  $W_i=\{w_{i,j}\}_{j=1}^d$ and $U_i=\{u_{i,j}\}_{j=1}^d$.

\begin{figure}[h]
    \centering

\begin{tikzpicture}[scale=1.2]
    \filldraw[black] (0,0) circle (0.05cm) node[anchor=north] {$0$};
    \foreach \s in {-1, 1} {
        \draw[thick] (0,0) -- (\s,0);
        \draw (\s*1.3, 0) node {$\cdots$};
        \draw[thick] (\s*1.6,0) -- (\s*2.6,0);
        \ifthenelse{\s=1}{\filldraw[black] (\s*2.6,0) circle (0.05cm) node[anchor=north east] {$n$}}{\filldraw[black] (\s*2.6,0) circle (0.05cm) node[anchor=north west] {$-n$}};
    }

    \def\n{6}
    \def\xx{4.5}
    \foreach \sc in {-1, 1} {
        \foreach \high in {1, -1, 2, -2} {
            \def\cent{\sc*\xx};
            \node[circle,minimum size=0.75 cm] at (\cent,\high) (b) {};
            \foreach\x in{1,...,\n}
            {
              \node[inner sep=0cm, minimum size=0.05cm, draw,circle, fill=black] (n-\x) at (b.{360/\n*\x}) {} ;
            }
            \foreach\x in{1,...,\n}{
                \foreach\y in{\x,...,\n}{
                    \ifnum\x=\y\relax\else
                    \draw[thick] (n-\x) edge[-] (n-\y);
                    \fi
                }
            \draw[red] (\sc*2.6, 0) -- (n-\x);
            }
        }
        \draw (\sc*\xx,0.1) node {$\vdots$};
    }
\end{tikzpicture}
    \caption{A picture of $H_{6,n}$. There are $m=\frac{n}{6}$ cliques of size $6$ on each side. The vertices in the top left clique are $w_{1,1},\ldots,w_{1,6}$.}
    \label{fig:Hdn}
\end{figure}
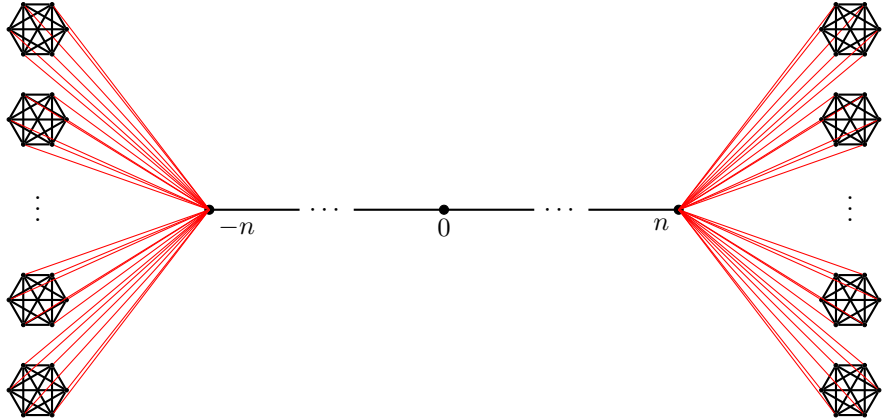
%


We are now ready to state our next theorem, which holds for all $p>1$ and shows that 
 \eqref{eq:p D lower} holds for $H_{d,n}$   for all $1<p\leq 2$, provided that $D>D(p)$, as well as for $2<p\leq 3$ when $D$ is of order $n$. This  will complete the proof of the converse statement in Theorem \ref{th:main lpbound} for these cases.

    \begin{thm}
        Run the $\ell^p$-energy minimizing dynamics on the graph $H_{d,n}$ with initial opinions $f_0$ defined as follows:  $f_0=0$  on all vertices $w_{i,j}$ and on $[-n,0)$, at the center $f_0(0)=1/2$,  and     $f_0=1$ on  all $u_{i,j}$ and on $(0,n]$.  Then for every large enough $n$ and $d$, every updating sequence, and every $t<\frac{1}{25e} 2^\frac{-p}{p-1}n^\frac{2p-1}{p-1}d^\frac{1}{p-1}$,  we have $\osc(f_t)\geq \frac{1}{2}$.
    \end{thm}

\begin{proof}
Let $\delta =   e\left(\frac{2}{nd}\right)^{\frac{p}{p-1}} $
(this choice of $\delta$ will emerge from the proof). 
Define a new initial profile $h_0:V\rightarrow [0,1]$ by 
    \begin{eqnarray*}
        &&h_0(w_{i,j}) = (j-1) \delta \quad 1\leq i\leq m\, , 1\leq j \leq d ,\\
        &&h_0(-n)=e^{-1} d^\frac{p}{p-1}\delta=  (2/n)^{\frac{p}{p-1}}\, ,\\
        &&h_0(i) = h_0(-n) + (i+n)\frac{1-h_0(-n)}{n} \quad -n<i\leq 0 \, , \\
        &&h_0(i) = 1 \quad 0<i\leq n \, ,\\
        &&h_0(u_{i,j}) = 1 \quad 1\leq i\leq m\, , 1\leq j \leq d .\\
    \end{eqnarray*}
    
Fix an updating sequence. As in the previous example, we let $h_t$ denote the evolution of $h_0$ under the $\ell^p$-energy minimizing dynamics with the same update sequence. 
Since $h_0\geq f_0$, it is enough to prove that 

\begin{equation}\label{eq:large D main}  \forall t<\frac{1}{25e} 2^\frac{-p}{p-1}n^\frac{2p-1}{p-1}d^\frac{1}{p-1},\quad  \forall i\le m, \quad  \min_j h_t(w_{i,j})<\frac{1}{4} \, ,\end{equation}  
as this would imply that $\osc(f_t)\geq \frac{1}{2}$ for such $t$.

We will need  some further notation. 
Given  a function $h:V\rightarrow \R$, we denote by $\widetilde{h}$ the function obtained by re-ordering the values of $h$ on each clique $W_i$, so that the sequence $\{\widetilde{h}(w_{i,j})\}_{j=1}^d$ is nondecreasing in $j$ for every $1\leq i\leq m$. 
We denote by $k_i(t)$ the number of updates in $W_i$ by time $t$. 

We will deduce (\ref{eq:large D main}) from the following 
claim:
\begin{claim}
          For every integer $t\in \big[0, \frac{m}{6e}d^{\frac{p}{p-1}}\big]$, we have
    \begin{itemize}
        \item [(a)]
        For all $0\leq j\leq n$, 
\begin{equation}\label{eq:star}
            h_t(-j)\leq h_0(-j)  \, ;
        \end{equation}
        \item [(b)] For all $1\leq i\leq m$ and all $1\leq j\leq d$,
\begin{equation}\label{eq:starstar}
\widetilde{h_t}(w_{i,j})\leq \min\bigl(h_0(-n), \, h_0(w_{i,j})+k_i(t)\delta \bigr) \, .
\end{equation}        
    \end{itemize}
\end{claim}
\begin{proof}
    We will prove both clauses together by induction on $t$.  For $t=0$ the claim is trivial, so we move to the induction step and assume that $t+1 \le  \frac{m}{6e}d^{\frac{p}{p-1}}$ and the inequalities \eqref{eq:star} and \eqref{eq:starstar} hold at time $t$. 
    
    We start by showing that \eqref{eq:starstar} holds for time $t+1$. The inequality
    \begin{equation} \label{clear1}
     \forall i,j \quad  \; h_{t+1}(w_{i,j}) \le h_0(-n)  
    \end{equation}
    is immediate from the induction hypothesis and monotonicity of the dynamics. 
    
    We may assume that the vertex $v_{t+1}$ updated at time $t+1$ is in  $W_i$ for some $i \le m$, since otherwise \eqref{eq:starstar} clearly continues to hold.
    Reorder  $W_i=\{\zeta_j\}_{j=1}^d$ so that $j \mapsto h_{t}(\zeta_j)$ is nondecreasing, i.e., for all $j \in \{1,\ldots,d\}$, we have
    \begin{equation} \label{enumi0}
    h_{t}(\zeta_j)=\widetilde{h_{t}}(w_{i,j}) \,. 
    \end{equation} 
We also reorder  $W_i=\{z_j\}_{j=1}^d$ so that $j \mapsto h_{t+1}(z_j)$ is nondecreasing and  we have
    \begin{equation} \label{enumi}
    \forall j \in \{1,\ldots,d\}, \quad \;h_{t+1}(z_j)=\widetilde{h_{t+1}}(w_{i,j}) \,. 
    \end{equation} 
    
    Suppose that $v_{t+1}=z_\ell \in W_i$. For   $j\in  \{1,\ldots,\ell-1\}$
    we have 
    $z_j \in \{\zeta_j,  \zeta_{j+1}\}$,  so 
    by the induction hypothesis and \eqref{enumi0},
    $$h_{t+1}(z_j)  \le h_t(\zeta_{j+1})\le h_0(w_{i,j+1})+k_i(t)\delta=h_0(w_{i,j})+k_i(t+1)\delta \,.$$
    For   $j\in  \{\ell+1,\ldots,d\} $
    we have $z_j\in \{\zeta_j,  \zeta_{j-1}\}$,
    whence 
    $$h_{t+1}(z_j)  \le h_t( \zeta_j)\le h_0(w_{i,j})+k_i(t+1)\delta\,.$$
     In view of \eqref{clear1} and \eqref{enumi}, this verifies   \eqref{eq:starstar}  at time $t+1$ for all $j \ne \ell$. 
    
    Similarly, if $j=\ell<d$, then
    $z_{\ell+1}\in \{\zeta_\ell,  \zeta_{\ell+1}\}$, so
    $$h_{t+1}(z_\ell) \le h_{t+1}(z_{\ell+1}) \le h_{t}(\zeta_{\ell+1}) \le h_0(w_{i,\ell+1})+k_i(t)\delta = h_0(w_{i,\ell})+k_i(t+1)\delta \,,$$ 
    verifying \eqref{eq:starstar}  at time $t+1$ for   $j = \ell$ if $\ell<d$.

The only remaining case is $j=\ell=d$, and for this case we must check that
\begin{equation} \label{eq:l-equals-d} h_{t+1}(z_d) \stackrel{?}{\leq} h_0(w_{i,d})+k_i(t+1)\delta \,.
\end{equation}
When verifying this, we may assume that $h_0(w_{i,d})+k_i(t+1)\delta < h_0(-n)$, since otherwise   \eqref{clear1} implies that \eqref{eq:l-equals-d} holds. 
    By the superharmonicity criterion \eqref{eq:psuper}, it is enough to prove that
    \begin{eqnarray*}  
\sum_{j=2}^{d}\bigl(h_0(w_{i,d})+k_i(t+1)\delta -\widetilde{h_t}(w_{i,j})\bigr)^{p-1} 
  \stackrel{?}{\geq} \bigl(h_t(-n)-h_0(w_{i,d})-k_i(t+1)\delta \bigr)^{p-1} .
    \end{eqnarray*}
Recalling the induction hypotheses
\eqref{eq:star} and \eqref{eq:starstar}, as well as   the definition of $h_0$, the preceding inequality   would follow  if we show that
\begin{equation}
    \sum_{j=2}^{d}\bigl((d-j+1)\delta \bigr)^{p-1} \stackrel{?}{\geq}  h_0(-n)^{p-1}\, .
\end{equation}
    Comparing the sum to an integral, it suffices to verify that
    \begin{equation}
        \frac{(d-1)^p}{p}\delta^{p-1}\stackrel{?}{\geq} \Bigl(e^{-1}d^\frac{p}{p-1}\delta \Bigr)^{p-1} \, .
    \end{equation}
    It is therefore enough to prove that
    \begin{equation}
        \frac{(d-1)^p}{p}\stackrel{?}{\geq}  {e^{1-p}}d^p \, ,
    \end{equation}
    which is equivalent to 
    \begin{equation}
        \left(\frac{d-1}{d}\right)^p\stackrel{?}{\geq}pe^{1-p} \, .
    \end{equation}
    The last inequality holds for all $d>d(p)$, since   $pe^{1-p}<1$ for $p>1$.

    We now move to show that \eqref{eq:star} holds at time $t+1$. This only has to be checked if at time $t+1$ we update one of the vertices $-n,\ldots, 0$, and only at the updated vertex. For any $j\in \{-n+1,\dots,0\}$, this holds by monotonicity of the dynamics and the definition of $h_0$, so we are left with the only interesting case of updating vertex $-n$. Thus we need to show that $h_{t+1}(-n)\leq h_0(-n)$. For this to hold it is enough to check that 
   \begin{equation} 
       \sum_{i=1}^m \sum_{j=1}^d \Bigl(h_0(-n)  - \widetilde{h_t}(w_{i,j})\Bigr)^{p-1} 
 \stackrel{?}{\geq} \big(h_t(1-n)-h_0(-n)\big)^{p-1} \,.
   \end{equation}
    By the induction hypothesis,  it suffices to prove that 
\begin{equation}\label{eq: 3}
        \sum_{i=1}^n \sum_{j=1}^d \max \bigl(0 \,,   h_0(-n)-h_0(w_{i,j})-k_i(t) \delta \bigr)^{p-1}\stackrel{?}{\geq} n^{1-p} \,.
    \end{equation}
    Let $S_t= \{i\leq m \, : \, k_i(t)\leq \frac{2t}{m} \}$.
    For  $d>d(p)$,  we have
    $d+\frac{1}{3e}d^{\frac{p}{p-1}} <\frac{1}{2e} d^{\frac{p}{p-1}}$, so if also $t< \frac{m}{6e}d^{\frac{p}{p-1}}$ and $i \in S_t$, then  
    $$ h_0(w_{i,j})+k_i(t) \delta \le h_0(w_{i,j})+\frac{2t}{m} \delta \le \frac{1}{2} h_0(-n) \,.$$
     Thus to verify \eqref{eq: 3}, it is enough to check   that
    \begin{equation}
        \sum_{i\in S_t} d\left( \frac{1}{2} h_0(-n)\right)^{p-1}\stackrel{?}{\geq} n^{1-p} \, .
    \end{equation}
    
     Since $\sum_{1\leq i\leq m}k_i(t)\leq t$, we have $|S_t|\geq \frac{m}{2}$, so it suffices to show that
    \begin{equation}
        \frac{md}{2} \Bigl(\frac12 h_0(-n) \Bigr)^{p-1}\stackrel{?}{\geq} n^{1-p} \,. 
    \end{equation}
    Recalling that $md=n$ and that $h_0(-n)=\frac{1}{e} d^\frac{p}{p-1}\delta$, the last inequality will follow for large $d$ if we show that
    \begin{equation}
        \frac{n}{2}\left(\frac{1}{2e}d^\frac{p}{p-1}\delta\right)^{p-1}\stackrel{?}{\geq} n^{1-p} \, .
    \end{equation}
    This, indeed, holds for  
    $\delta =   e\left(\frac{2}{nd}\right)^{\frac{p}{p-1}}$.   The proof of the claim is complete. \end{proof}
Let $t_*:= \Big\lfloor\frac{m}{6e}d^{\frac{p}{p-1}} \Big\rfloor \,.$
The claim implies that $\widetilde{h_{t_*}} \leq h_0 + h_0(-n)$ (pointwise). Another induction gives that for all integers $\ell \ge 1$, we have
\[ \widetilde{h_{\ell t_*}}\leq h_0+\ell h_0(-n) \,.\]  

We deduce that 
$$  \forall t<\left\lfloor\frac{1}{4h_0(-n)} \right\rfloor t_* \, \,,\quad  \forall i\le m \,,\quad \; \min_j h_t(w_{i,j})<\frac{1}{4} \, .$$ Since $$\left\lfloor\frac{1}{4h_0(-n)} \right\rfloor t_* \geq \frac{1}{25e} 2^\frac{-p}{p-1}n^\frac{2p-1}{p-1}d^\frac{1}{p-1} \, $$ for $d>d(p)$, the theorem is proved.


\end{proof}

\subsection{Lower bound for $2<p<3$.}
We may assume that $D\geq 20$, as for smaller $D$ we can apply Theorem \ref{thm:cycle1}. 
In this subsection we analyze the $\ell^p$-energy minimizing dynamics on the accordion graph depicted in Figures \ref{fig:accordion} and \ref{fig:accordion large}.
The first ingredient we need is the following lemma. It shows that when the opinions at the neighbours of a vertex $w$ form two arithmetic progressions with the same gap and length, the value that minimizes the energy at $w$ does not depend on $p$. 

\begin{lem}\label{le:two arithmetic progressions}
     Let $p>1$, and let $f$ be a function on $\{u_i\}_{i=1}^d \cup \{ w_i\}_{i=1}^d$ that satisfies $f(u_i)=f(u_1)+(i-1)a$ and $f(w_i)=f(w_1)+(i-1)a$ for every $1<i\leq d$. Then 
    $\psi(z):= \sum_{i=1}^d \Bigl(|f(u_i)-z|^p + |f(w_i)-z|^p\Bigr)$ is minimized at $z_*=\frac{f(u_1)+f(w_d)}{2}$.
\end{lem}
\begin{proof}
    The function $\psi$ is convex and \begin{equation*}  
     \psi'(z)=p\sum_{i=1}^d \Bigl(|f(u_i)-z|^{p-1}\sign(f(u_i)-z) + |f(w_{d-i+1})-z|^{p-1}\sign (f(w_{d-i+1})-z)\Bigr)\, .
    \end{equation*} 
    Observe that for $z=z_*$, the two terms in the $i$th summand are of the same magnitude and opposite signs, because 
    $$ f(u_i)-z_* +f(w_{d-i+1})-z_* = f(u_1)+(i-1)a + f(w_d)-(i-1)a - 2z_*=0\, .$$
    Therefore $\psi'(z_*)=0$, as required. 
\end{proof}

Next, we analyze the dynamics on one side of the accordion.

\begin{figure}[h]
    \centering
\begin{tikzpicture}[scale=1.2]

  \foreach \x in {-3, -2, -1, 1, 2, 3} {
    \foreach \y in {-1.5, -1, -0.5, 0.5, 1, 1.5} {
        \foreach \yy in {-1.5, -1, -0.5, 0.5, 1, 1.5} {
          \draw (\x,\y) -- (\x+1,\yy);
        }
    }
  }

  \foreach \x in {-3, ..., 4} {
  \filldraw[black] (\x,0) circle (0cm) node {$\vdots$} ;
    \foreach \y in {-1.5, -1, -0.5, 0.5, 1, 1.5} {
        \filldraw[red] (\x,\y) circle (0.05cm) node[anchor=north] {} ;
    }
  }

  \foreach \y in {-1,-0.5,0,0.5,1}
    \filldraw[black] (0.5,\y) circle (0cm) node {$\cdots$} ;

  \node[anchor=west] (A) at (4,1.5) {$v_{m,d}$};
  \node[anchor=west] (A) at (4,1) {$v_{m,d-1}$};
  \node[anchor=west] (A) at (4,0.5) {$v_{m,d-2}$};
  \node[anchor=west] (A) at (4,-0.5) {$v_{m,3}$};
  \node[anchor=west] (A) at (4,-1) {$v_{m,2}$};
  \node[anchor=west] (A) at (4,-1.5) {$v_{m,1}$};

  \node[anchor=east] (A) at (-3,1.5) {$v_{-m,d}$};
  \node[anchor=east] (A) at (-3,1) {$v_{-m,d-1}$};
  \node[anchor=east] (A) at (-3,0.5) {$v_{-m,d-2}$};
  \node[anchor=east] (A) at (-3,-0.5) {$v_{-m,3}$};
  \node[anchor=east] (A) at (-3,-1) {$v_{-m,2}$};
  \node[anchor=east] (A) at (-3,-1.5) {$v_{-m,1}$};
\end{tikzpicture}
    \caption{The graph $H=H(m,d)$ analyzed in Lemma \ref{lem:half graph} consists of $2m+1$ anti-cliques $\{V_k\}_{-m\leq k \leq m}$. Each vertex in $V_k$ for $|k|<m$ is connected to all vertices in $V_{k\pm 1}$. }
    \label{fig:half-accordion}
\end{figure}
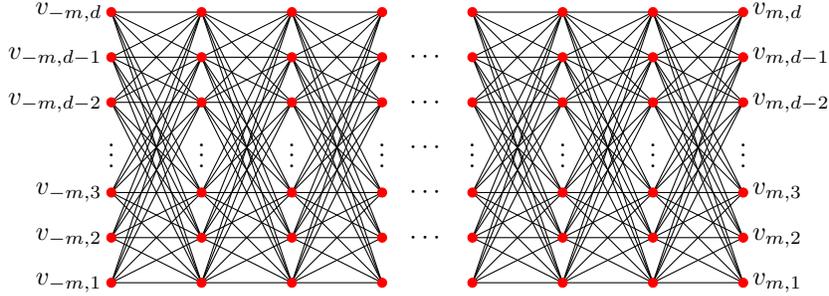

Let $m> 8$. Consider the graph $H=H(m,d)$, depicted in Figure \ref{fig:half-accordion}, which consists of $2m+1$ anti-cliques $\{V_k\}_{-m\leq k \leq m}$ where $V_k=\{v_{k,j}\}_{j=1}^d$. For every $k \in \{-m,\dots, m-1\}$,    each vertex in $V_k$  is adjacent to all vertices in $V_{k+1}$. 

\noindent \textbf{Notation}. Given a function $F$ on $[-m,m]$ and $-m<k<m$, we denote 
$$\overline{F}(k):=\frac{F(k-1)+F(k+1)}{2} \, .$$ 
\begin{lem}\label{lem:half graph}
     Fix some update sequence $\{(k_s,j_s)\}_{s\geq 1}$ for the graph $H$ defined above, with $|k_s|<m$ for all $s$.  Let $\delta=\frac{1}{m^2(d-1)}$, and run
the $\ell^p$-energy minimizing dynamics on $H$ using this update sequence, starting with initial opinions $$\varphi_0(v_{k,j})=\frac{k^2}{2m^2}+(j-1)\delta \, .$$
    
Let $M_t(k):=\max_{1\leq j\leq d}\varphi_t(v_{k,j})$ and $M_t^*(k):= \max_{0 \le s \le t} M_s(k)$ for all integer $k \in [-m,m].$ Then,
    \begin{enumerate}[label=(\roman*)]
        \item  $M_t^*(k)$ is convex in $k\in [-m,m]$ for every $t\geq 0$ ;
        \item $M_t(k)\leq \frac{1}{4}+\frac{k^2}{2m^2}$ for all $t\leq\frac{1}{160}(d-1)m^3$ and all $-m\leq k\leq m$ ;
        \item $M_t(-m+1)\leq \frac{1}{2}-\frac{1}{8m}$ for all $t\leq \frac{1}{160}(d-1)m^3$.
    \end{enumerate}
\end{lem}

\medskip

\begin{proof}

We verify $(i)$   by induction on $t$. The base case $t=0$ is just convexity of a quadratic. If $(k_{t+1},j_{t+1})=(k,j)$, then  it suffices to check convexity of $M_t^*$ at $k$, since $M_t^*(k) \ge M_{t-1}^*(k)$. By the monotonicity of the dynamics, 
\begin{equation} \label{eq:fave} \varphi_{t+1}(v_{k,j}) \leq \frac{M_t(k-1)+M_t(k+1)}{2}\,,
\end{equation}
since if  all values of $\varphi$ on $V_{k\pm 1}$ are increased to $M_t(k\pm 1)$, respectively, then  \eqref{eq:fave}  becomes an equality. This, together with the induction hypothesis, proves (i). 

\medskip

\noindent{\em Proof of} $(ii)$: We will first construct recursively a sequence of functions $(\Phi_t)_{t \ge 0}$, and then show that they dominate $M_t$. 

\medskip

\noindent{\bf Step 1:} For $-m\leq k \leq m$, let $$\Phi_0(k):=\varphi_0(v_{k,d})=\frac{k^2+2}{2m^2} \, .$$
For $t\geq 0$, we define $\Phi_{t+1}$ as follows:\\
\begin{itemize}
    \item If $k=k_{t+1}$ and $\Phi_t(k)<\overline{\Phi}_t(k)+\frac{1}{2m^2}$, then $$\Phi_{t+1}(k) := \max \Bigl(\Phi_t(k)+\delta,\,\, \overline{\Phi}_t(k)-\frac{1}{2m^2}\Bigr) \, ;$$
    \item otherwise, $\Phi_{t+1}(k):=\Phi_t(k)$.
\end{itemize}

\medskip

\noindent{\bf Step 2:}   Next we show that $M_t\leq \Phi_t$ for all $t$.  We will deduce this from the following stronger statement:

    Let $\widetilde{\varphi_t}$ be obtained from $\varphi_t$ by sorting the values of $\varphi_t$ within each $V_k$, so that $\widetilde{\varphi_t}(v_{k,j})$ is nondecreasing in $j$. Then we claim that
    \begin{equation} \label{eq:fstar}
        \forall k,j \quad \, \widetilde{\varphi_t}(v_{k,j})\leq \Phi_t(k)-(d-j)\delta \, .
    \end{equation}
    (Setting $j=d$ will yield that $M_t\leq \Phi_t$).
    
    We will prove \eqref{eq:fstar} by induction on $t$. 
For $t=0$, it  trivially holds. Next, we assume that \eqref{eq:fstar} holds for $t$ and   deduce that it also holds for $t+1$.  It is enough to check this for $k=k_{t+1}$.
    By the induction hypothesis, 
 for all $1\leq j \leq d$, we have $$\widetilde{\varphi_t}(v_{k\pm 1,j}) \le \Phi_t(k\pm 1)-(d-j)\delta\, ,$$ so by 
   Lemma \ref{le:two arithmetic progressions} and monotonicity, we infer that 
    \begin{equation} \label{bigbound} \varphi_{t+1}(v_{k,j_{t+1}})\leq \overline{\Phi}_t(k)-\frac{(d-1)\delta}{2}=\overline{\Phi}_t(k)-\frac{1}{2m^2} \, .
    \end{equation}
Reorder  $V_k=\{\zeta_j\}_{j=1}^d$ so that $j \mapsto \varphi_{t}(\zeta_j)$ is nondecreasing, i.e., 
    \begin{equation} \label{enumi2}
    \forall j \in \{1,\ldots,d\}, \quad \; \varphi_{t}(\zeta_j)=\widetilde{\varphi_{t}}(v_{k,j}) \,. 
    \end{equation} 
We also reorder  $V_k=\{z_j\}_{j=1}^d$  so that $j \mapsto \varphi_{t+1}(z_j)$ is nondecreasing and  we have
    \begin{equation} \label{enumi3}
    \forall j \in \{1,\ldots,d\}, \quad \;\varphi_{t+1}(z_j)=\widetilde{\varphi_{t+1}}(v_{k,j})  \,. 
    \end{equation} 
Suppose that $v_{k,j_{t+1}}=z_{j^*}$.
We consider two cases.  
\medskip

\noindent{\sf Case A:} If 
\begin{equation} \label{bigif}
    \Phi_t(k)\geq \overline{\Phi}_t(k)+\frac{1}{2m^2} \,,
\end{equation}
 then $\Phi_{t+1}(k)=\Phi_t(k)$, and \eqref{bigbound} implies that 
 $$   \widetilde{\varphi_{t+1}}(z_{j^*})=\varphi_{t+1}(v_{k,j_{t+1}})\leq
  {\Phi}_t(k)-\frac{1}{m^2}={\Phi}_t(k)-(d-1)\delta \,.$$
  This verifies the induction step in Case A for $j \le j^*$. On the other hand,  for $j>j^*$ we have 
  $z_j\in \{\zeta_j, \zeta_{j-1}\}$, so 
  $\varphi_{t+1}(z_j)\le \varphi_{t}( \zeta_j)$ and  \eqref{eq:fstar} at time $t+1$ follows from the induction hypothesis.
  
  \medskip
  
\noindent{\sf Case B:} If \eqref{bigif} does not hold, then (recalling that $k=k_{t+1}$) we have $$\Phi_{t+1}(k) = \max \Bigl(\Phi_t(k)+\delta,\,\, \overline{\Phi}_t(k)-\frac{1}{2m^2}\Bigr) \,.$$
To finish the induction step in this case, we 
consider separately three subcases:
 when $j \ne j^*$, when  
$j=j^*<d$ and when $j=j^*=d.$

First,  for all $j \ne j^*$ we have $z_j\in \{\zeta_j,\zeta_{j \pm 1}\}$, so
\begin{equation*}    \varphi_{t+1} (z_j)\leq  \varphi_t (\zeta_{j+1}) \leq \Phi_t(k)-(d-j-1)\delta  \le \Phi_{t+1}(k)-(d-j)\delta \, .
\end{equation*}

Second, if $j^*<d$, then $z_{j^*+1}\in \{\zeta_{j^*},\zeta_{j^*+1}\}$, so 
$$\varphi_{t+1}(z_{j^*} ) \le 
\varphi_{t}(\zeta_{j^*+1} ) \le 
\Phi_t(k)-(d-j^*-1)\delta  \le \Phi_{t+1}(k)-(d-j^*)\delta \,.$$

Finally, if $j^*=d$, then   \eqref{bigbound} yields that
$$  \varphi_{t+1}(z_{j^*} )  \leq \overline{\Phi}_t(k) -\frac{1}{2m^2} \leq \Phi_{t+1}(k) \, .$$

This completes the induction step, and establishes \eqref{eq:fstar}.

\medskip

\noindent{\bf Step 3:} To complete the proof of part (ii), it remains to show that the inequality $\Phi_t(k)\leq \frac{1}{4}+\frac{k^2}{2m^2}$ holds for all $t\leq \frac{1}{160}m^3d$. 
 Define $$g_t(k):=\Phi_t(k)-\Phi_0(k)=\Phi_t(k)-\frac{k^2+2}{2m^2} \, .$$ This implies that $$\overline{g}_t(k)=\overline{\Phi}_t(k)-\frac{k^2+3}{2m^2}\, .$$ 
 Thus if $k=k_{t+1}$ and $g_t(k)<\overline{g}_t(k)+\frac{1}{m^2}$, then \begin{equation}\label{eq:g inc} g_{t+1}(k) = \max \Bigl(g_t(k)+\delta, \,\overline{g}_t(k)\Bigr) \, ,\end{equation}
and otherwise  $g_{t+1}(k)=g_t(k)$.
Next, we will verify by induction on $t$ that for all $-m<k<m$, \begin{equation} \label{eq:crazy} \Delta g_t(k) := \overline{g}_t(k) - g_t(k) \geq \frac{-1}{m^2}-\delta \geq \frac{-2}{m^2} \, .\end{equation}
The base case $t=0$ is clear. For the inductive step, we assume \eqref{eq:crazy} and show that it also holds when $t$ is replaced by $t+1$. Since $g_t \le g_{t+1}$, we can focus on the case $k=k_{t+1}$; in this case, the inequality
$\Delta g_{t+1}(k)   \geq \frac{-1}{m^2}-\delta$ follows directly from  \eqref{eq:g inc} together with the lines before and after it. 

Let $h_t$ denote the concave envelope of $g_t$ (as in \eqref{concenv}). Then $h_0\equiv 0$, and 
\begin{equation}\label{eq:new delta h}
    0\geq \Delta h_t \geq \frac{-2}{m^2} \, .
\end{equation}
By the concavity of $h_t$ and \eqref{eq:g inc}, 
\begin{equation}\label{eq:new h increment}
    0\leq h_{t+1}(k_{t+1})-h_t(k_{t+1}) \leq \delta \, .
\end{equation}
Define the function 
$$h^*_{t+1}(k):=\begin{cases}
    h_t(k) \quad &k\neq k_{t+1}\,, \\
    h_{t+1}(k) \quad &k=k_{t+1} \,.
\end{cases}$$

 Extend the functions $h_t$ and $h^*_{t+1}$ to functions on $[-n,n]$ by linear interpolation. As before, let $\ell(h)$ denote  the length of the graph of $h$.
Then
\begin{equation}\label{eq:delta length}
    \ell(h_{t+1})\leq \ell(h^*_{t+1})\leq \ell(h_t)+\frac{8\delta}{m^2}\, ,
\end{equation}
with the last inequality following from Lemma \ref{le: geometric lemma} using \eqref{eq:new delta h} and \eqref{eq:new h increment}.


Finally,  we  note that $\ell(h_0)=2m$, and that if for some $k$ and $t$ $$\Phi_t(k)\geq \frac{1}{4}+\frac{k^2}{2m^2}\, ,$$ then $h_t(k)\geq \frac{1}{4}-\frac{1}{m^2}$, and  
$$\ell(h_t)\geq 2\sqrt{m^2+(1/4-1/m^2)^2} \ge 2m+\frac{1}{20m} \,,$$
since $m>8$. Thus by \eqref{eq:delta length}, $$
t > \frac{1/(20m)}{8\delta/m^2} =\frac{m^3 (d-1)}{160}\, .$$
This concludes the proof of (ii).
\item From (ii) we deduce that for $t\leq \frac{m^3(d-1)}{160}$ we have  $M_t^*(0) \le \frac14$, whence by (i),
\begin{eqnarray*}
 M_t(-m+1)  &\le&  M_t^*(-m+1)   \\ &\le&  
\Bigl(1-\frac1m\Bigr)M_t^*(-m)+\frac{1}{m} M_t^*(0)   \\&\le&  \Bigl(1-\frac1m\Bigr)\Bigl(\frac12+\frac1{m^2}\Bigr)+\frac{1}{4m}    \le \frac12-\frac{1}{8m} \,. 
\end{eqnarray*}


\end{proof}

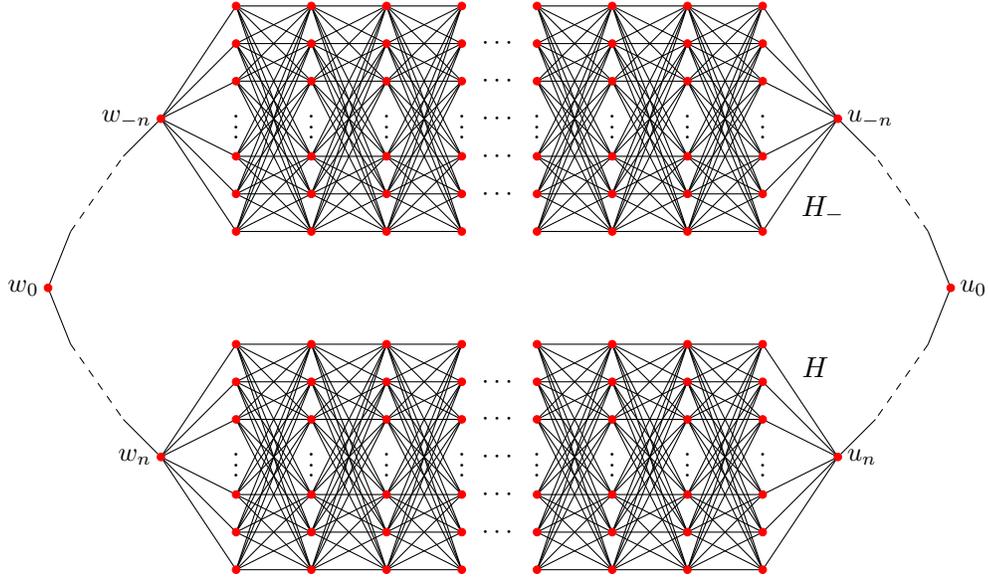
\begin{figure}[h]
    \centering
\begin{tikzpicture}
  \def\goody{2.25};
  \foreach \sign in {1,-1} {
  \def\basey{ \sign*\goody };
  \foreach \x in {-3, -2, -1, 1, 2, 3} {
    \foreach \y in {-1.5, -1, -0.5, 0.5, 1, 1.5} {
        \foreach \yy in {-1.5, -1, -0.5, 0.5, 1, 1.5} {
          \draw (\x,\basey+\y) -- (\x+1,\basey+\yy);
        }
    }
  }

  \foreach \y in {-1.5, -1, -0.5, 0.5, 1, 1.5} {
    \draw (-4,\basey) -- (-3,\y+\basey);
  }
  \foreach \y in {-1.5, -1, -0.5, 0.5, 1, 1.5} {
    \draw (5,\basey) -- (4,\y+\basey);
  }
  \draw (-4.5,\basey-\sign*0.5) -- (-4,\basey);
  \draw (5.5,\basey-\sign*0.5) -- (5,\basey);
  \draw (-5.5,0) -- (-5.2,\sign*0.75);
  \draw (6.5,0) -- (6.2,\sign*0.75);
  \draw[dashed] (6.2,\sign*0.75) -- (5.5,\basey-\sign*0.5);
  \draw[dashed] (-5.2,\sign*0.75) -- (-4.5,\basey-\sign*0.5);

  \foreach \x in {-3, ..., 4} {
  \filldraw[black] (\x,\basey) circle (0cm) node {$\vdots$} ;
    \foreach \y in {-1.5, -1, -0.5, 0.5, 1, 1.5} {
        \filldraw[red] (\x,\basey+\y) circle (0.05cm) node[anchor=north] {} ;
    }
  }

  \foreach \y in {-1,-0.5,0,0.5,1}
    \filldraw[black] (0.5,\basey+\y) circle (0cm) node {$\cdots$} ;
  }

  \filldraw[red] (-4,\goody) circle (0.05cm) node[anchor=east, color=black] {$w_{-n}$} ;
  \filldraw[red] (5,\goody) circle (0.05cm) node[anchor=west, color=black] {$u_{-n}$} ;
  \filldraw[red] (-4,-\goody) circle (0.05cm) node[anchor=east, color=black] {$w_n$} ;
  \filldraw[red] (5,-\goody) circle (0.05cm) node[anchor=west, color=black] {$u_n$} ;

  \filldraw[red] (-5.5,0) circle (0.05cm) node[anchor=east, color=black] {$w_0$} ;
  \filldraw[red] (6.5,0) circle (0.05cm) node[anchor=west, color=black] {$u_0$} ;

  \draw (4.8, \goody-1.2) node {{\large $H_-$}};
  \draw (4.7, -\goody+1.2) node {{\large $H$}};
\end{tikzpicture}
    \caption{The {\bf accordion graph} $G_{d,n}$ which includes the graph $H$ (depicted in Figure \ref{fig:half-accordion}), together with another copy, $H_-$, above it. These are connected by two paths of length $2n$, one from the left and one from the right. The  two endpoints of the left (right) path are connected to all vertices in the leftmost (rightmost) anti-cliques of $H_-$ and $H$, respectively. }
    \label{fig:accordion large}
\end{figure}

Given $D \ge 20$ and $N \ge 60D$, write $d=\lfloor D/2\rfloor$ and let $n$ be the largest multiple of $d$ that satisfies
$10n \le N$.  Let $m:=n/d \ge 12$. Consider the {\bf accordion graph} $G_{d,n}$ depicted in Figure \ref{fig:accordion large}. It consists of two copies, $H_-$ and $H$, of $H(m,d)$, one above the other, connected by two paths of length $2n$, one from the left and one from the right.  The  two endpoints of the left (right) path are connected to all vertices in the leftmost (rightmost) anti-cliques of $H_-$ and $H$, respectively. We endow $G_{d,n}$ with the initial profile $f_0$ that takes value 0 on $V(H)\cup\{w_j\}_{j=1}^n \cup\{u_j\}_{j=1}^n $, takes value 1 on $V(H_-)\cup\{w_{-j}\}_{j=1}^n \cup\{u_{-j}\}_{j=1}^n $, and satisfies $f_0(w_0)=f_0(u_0)=1/2$.
\begin{thm}
     Fix $p \in [2,3]$ and run the $\ell^p$-energy minimizing dynamics on the accordion graph $G_{d,n}$ defined above with initial profile $f_0$. 
     If $d \ge 10$ and $n \ge 12d$ is a multiple of $d$, then for every updating sequence and every $t\leq \frac{1}{6400}n^3d^{\frac{3-p}{p-1}}$,  we have $\osc(f_t)\geq \frac{1}{2}$.
\end{thm}
\begin{proof}
    We will bound the opinion profiles $f_t$ that arise from $f_0$ by the profiles $\psi_t$ that arise from the initial profile $\psi_0 \ge f_0$ defined by 
    
    \begin{eqnarray*}
        &&\psi_0\equiv 1 \quad  \text{on} \quad  V(H_-)  \, , \\
        &&\psi_0(w_{-i}) = \psi_0(u_{-i})=1 \quad 0\leq  i\leq n \, ,\\
        &&\psi_0(w_n) = \psi_0(u_n) = \frac{1}{n}d^\frac{-1}{p-1}+\lambda \frac{m^2+2}{2m^2} \, ,\\
        &&\psi_0(w_i) = \psi_0(u_i)= \frac{n-i}{n}+\frac{i}{n}\psi_0(w_n) \quad 0\leq i < n \, ,\\
        &&\psi_0(v_{k,j}) = \lambda \Bigl(\frac{k^2}{2m^2}+ (j-1)\delta\Bigr) \quad -m\leq k\leq m, \quad 1\leq j\leq d \, ,
    \end{eqnarray*}
    where $\delta=\frac{1}{m^2(d-1)}$ as in the previous lemma, and 
    $$ \lambda = 16m n^{-1}d^{\frac{-2}{p-1}} = 16d^{-\frac{1+p}{p-1}} \, .
$$
    The value of $\psi_0(w_n)$ was chosen so that
    \begin{equation} d\Bigl(\psi_0(w_n)-\max_{v\in V_{-m}}\psi_0(v)\Bigr)^{p-1} =\Bigl(\frac{1}{n}\Bigr)^{p-1}\, \end{equation}
and similarly for $u_n$ and $V_m$ instead of $w_n$ and $V_{-m}$. 

    Next, we will verify by induction that for every $t\leq \frac{1}{160}(d-1)m^3$, the following inequalities hold:
    \begin{itemize}
        \item[($a_t$)]  $\psi_t(v_{k,j})\leq \lambda \Bigl(\frac{1}{4}+\frac{k^2}{2m^2}\Bigr)$ for all $-m < k < m$ and $1\leq j \leq d\,$;
        \item[($b_t$)] $\psi_t(v_{\pm m,j})\leq \psi_0(v_{\pm m,j})$ for all $1\leq j \leq d\, $ ;
        \item[($c_t$)] $\psi_t(w_i)\leq \psi_0(w_i)$ and $\psi_t(u_i)\leq \psi_0(u_i)$ for all $0\leq i\leq n$ . 
    \end{itemize}
The case $t=0$ is clear. To verify the induction step, assume that the inequalities $(a_s)$, $(b_s)$ and $(c_s)$ hold for all $s\leq t$ and that $t+1\leq \frac{1}{160}(d-1)m^3$. 
Then the inequality $(a_{t+1})$ follows from the second clause of Lemma \ref{lem:half graph}, because under $(b_s)_{s\leq t}$ we have $\psi_{t+1} \leq \lambda \varphi_{t'}$, where $t'$ is the number of updates in  the interior anti-cliques of $H$ (i.e., in $\cup_{|k|< m} V_k$) during the first $t+1$ steps. To verify $(b_{t+1})$, it is enough to check that for every $1\leq j\leq d$,
\begin{equation}\label{eq:psi0}
    \Bigl(\psi_t(w_n)-\psi_0(v_{-m,j})\Bigr)^{p-1}\stackrel{?}{\leq} \sum_{i=1}^d \Bigl(\psi_0(v_{-m,j})-\psi_t(v_{-m+1,i})\Bigr)^{p-1}
\end{equation}
(The case of $u_n$ and $v_{m,j}$ will follow similarly).

By the induction hypothesis,  
\begin{eqnarray*}
\psi_t(w_n)-\psi_0(v_{-m,j}) \leq \frac{1}{n}d^{\frac{-1}{p-1}} + \frac{\lambda}{m^2} = \\
\frac{1}{n}d^{\frac{-1}{p-1}}  + \frac{16d^{-\frac{p+1}{p-1}}}{m^2} 
\leq \frac{2}{n}d^\frac{-1}{p-1},\end{eqnarray*} 
where the last inequality holds since $d \ge 10$ and $m \ge 12$. 

Therefore, the LHS of \eqref{eq:psi0} is at most $2^{p-1} n^{1-p}d^{-1}$.  By Lemma \ref{lem:half graph} (iii), 
\begin{equation}
    \max_{1\leq i \leq d}\psi_t(v_{-m+1,i})\leq \lambda \Bigl(\frac{1}{2}-\frac{1}{8m} \Bigr) \, .
\end{equation}
Thus by our choice of $\psi_0$ and $\lambda$, the RHS of \eqref{eq:psi0} is at least 
\begin{equation} d \Bigl(\frac{\lambda}{8m}\Bigr)^{p-1} = d\Bigl(\frac{16d^{-\frac{p+1}{p-1}}}{8m}\Bigr)^{p-1} =2^{p-1}n^{1-p}d^{-1} \, ,
\end{equation}
so  \eqref{eq:psi0} holds. 

It remains to show $(c_{t+1})$. For $0\leq i<n$, this follows from $(c_t)$ by linearity of $\psi_0$. To show that $\psi_{t+1}(w_n)\leq \psi_0(w_n)$, it is enough to verify that
\begin{equation}\label{eq:endpoint}
\Bigl(\psi_t(w_{n-1})-\psi_0(w_n)\Bigr)^{p-1}\stackrel{?}{\leq} \sum_{i=1}^d \Bigl(\psi_0(w_n)-\psi_t(v_{-m,i})\Bigr)^{p-1} \,.
\end{equation}
The LHS is bounded above by $n^{1-p}$. The RHS is bounded below by $$d\Bigl(\psi_0(w_n)- \max_{v\in V_{-m}} \psi_0(v)\Bigr)^{p-1} = n^{1-p} \, .$$
This concludes the induction step, so the inequalities $(a_t),(b_t),(c_t)$ hold for all $t\leq t^*$, where $t^*:= \Big\lfloor\frac{1}{160}(d-1)m^3\Big\rfloor \ge \frac{1}{200}dm^3$.
(Here we used our assumptions that $ d=\lfloor D/2 \rfloor \ge 10$ and $m \ge 12$.) 

We deduce that $\max_{t \le t^*} \psi_{t}\leq \psi_0+\frac{\lambda}{4}$, and therefore $\max_{t \le Lt^*}  \psi_{t}\leq \psi_0+\frac{L\lambda}{4}$ for all integer $L\geq 1$.
In particular,  
  for integer  $L\leq   \frac{1}{\lambda} $, we have $\max_{t \le t^*} \psi_{t }(v_{0,1})\leq\frac{1}{4}$.
Therefore, $\min f_t \leq \min \psi_t \leq\frac{1}{4} \quad  \text{ for all }\quad  t\leq \frac{t^*}{2\lambda}$. (Note that $1/\lambda>2$.)
By symmetry, $\max f_t \geq \frac{3}{4}$ for all $t\leq \frac{t^*}{2\lambda}$, so $\osc(f_t) \ge 1/2$ for these values of $t$. 

Since $\frac{t^*}{2\lambda} \ge \frac{1}{6400}d^\frac{2p}{p-1}m^3=\frac{1}{6400}d^{\frac{3-p}{p-1}}n^3 \, ,$ this concludes the proof. 

\end{proof}
  
    
\section{Open problems and concluding remarks}\label{s:open}
 The upper and lower bounds in our main theorems suggest that the $\epsilon$-consensus time should be monotone decreasing in $p$. 
\begin{question} 
       Is there an absolute constant $C$ such that for every choice of integer $n> 1$,     connected graph $G$ with $n$ vertices,   initial profile $f_0$and $\epsilon>0$, for every   $1<p_1<p_2<\infty$, we have
       $$\E[\tau_{p_2}(\epsilon)]\leq C\E[\tau_{p_1}(\epsilon)] \, ?$$
              
\end{question}

\bigskip

Similar to the Lipschitz learning dynamics, the dynamics \eqref{eq:pdef} for $1<p<\infty$ can also be considered with prescribed boundary values. 
It is not hard to show that given boundary values $f_0|_B$, the dynamics converge to a limiting profile which is the unique $p$-harmonic extension $h_p$ of the values on the boundary to the whole graph (using the same energy considerations discussed in subsection \ref{sub:basic}). 



It is natural to ask how quickly the dynamics converge  towards the unique $p$-harmonic extension $h_p$. To this end, we define the $\epsilon$-approximation time by:

\begin{equation}
    \tau_p^*(\epsilon):= \min\{t\geq 0 :\, \|f_t -h_p\|_\infty \leq \epsilon\}\, .
\end{equation}

\begin{question}
    Given $1<p<\infty$, what can be said about the asymptotics of $\tau_p^*(\epsilon)$? 
    \end{question}
     
    For $p>2$,  we expect that there are absolute constants $C_p,\beta^*_p$ such that for every connected graph $G=(V\cup B,E)$ with $|V|=n$ inner vertices, every initial profile $f_0:V\cup B \rightarrow [0,1]$ and every $\epsilon>0$, $$\E[\tau_p^*(\epsilon) ] < C_p n^{\beta^*_p} \log\frac{1}{\epsilon}\, ;$$
we conjecture that this holds with $\beta^*_p=\beta_p$.

For $1<p<2$, a simple examples with $|V|=|B|=2$ shows that the dependence on $\epsilon$ is not logarithmic: Let $G$ be the complete graph on 4 vertices, and suppose that $f_0(b)=1$ for one vertex $b \in B$ while $f_0$
vanishes on the other three vertices. Then a  direct calculation shows that 
$$\forall p \in (1,2), \quad  \tau_p^*(\epsilon) \ge c_p \epsilon^{\frac{p-2}{p-1}} \, ,$$
where $c_p>0$ depends only on $p$.

\section*{Acknowledgments}
We thank Guy Blachar for his help with creating the figures in the paper. G.A. was supported by ISF grant $\#957/20$.

\bibliographystyle{plain}
\bibliography{lp_bibliography}

\end{document}